\documentclass[a4paper]{article}

\usepackage{a4wide}

\RequirePackage{amsmath,amsfonts,amssymb,amsthm,booktabs,color,epsfig,graphicx,hyperref,url,bbm}
\RequirePackage{dsfont}

\RequirePackage[numbers]{natbib}

\theoremstyle{plain}
\newtheorem{theorem}{Theorem}

\newtheorem{rem}{Remark}
\newtheorem{proposition}{Proposition}
\newtheorem{lemma}{Lemma}

\theoremstyle{definition}

\DeclareMathOperator\card{card}
\DeclareMathOperator\supp{supp}

\title{Parametric convergence rate of a non-parametric estimator in multivariate mixtures of power series distributions under conditional independence}

\author{
  Fadoua Balabdaoui\thanks{Department of Mathematics, ETH Zurich, Zurich, Switzerland, email: fadouab@ethz.ch}
  \and
  Harald Besdziek\thanks{Department of Mathematics, ETH Zurich, Zurich, Switzerland, email: harald.besdziek@stat.math.ethz.ch}
  \and
  Yong Wang\thanks{Department of Statistics, University of Auckland, Auckland, New Zealand, email: yongwang@auckland.ac.nz}
}

\date{September 5, 2025}

\begin{document}

\maketitle

\begin{abstract}
  The conditional independence assumption has recently appeared in a
  growing body of literature on the estimation of multivariate
  mixtures. We consider here conditionally independent multivariate
  mixtures of power series distributions with infinite support, to
  which belong Poisson, Geometric or Negative Binomial mixtures. We
  show that for all these mixtures, the non-parametric maximum
  likelihood estimator converges to the truth at the rate
  $(\log (nd))^{1+d/2} n^{-1/2}$ in the Hellinger distance, where $n$
  denotes the size of the observed sample and $d$ represents the
  dimension of the mixture. Using this result, we then construct a new
  non-parametric estimator based on the maximum likelihood estimator
  that converges with the parametric rate $n^{-1/2}$ in all
  $\ell_p$-distances, for $p \ge 1$. These convergences rates are
  supported by simulations and the theory is illustrated using the
  famous V\'{e}lib dataset of the bike sharing system of Paris. We
  also introduce a testing procedure for whether the conditional
  independence assumption is satisfied for a given sample. This
  testing procedure is applied for several multivariate mixtures, with
  varying levels of dependence, and is thereby shown to distinguish
  well between conditionally independent and dependent
  mixtures. Finally, we use this testing procedure to investigate
  whether conditional independence holds for V\'{e}lib dataset.
\end{abstract}


\section{Introduction}

\subsection{General scope and existing literature}

Mixture distributions are very important in statistical modeling and
are used in a variety of applications such as engineering, economics,
finance, biology and medicine, etc. Their wide applicability stems
essentially from the additional degree of freedom they can provide in
fitting datasets; see \cite{Lindsay}, \cite{McLachlan} and
\cite{Titterington}. Another important feature of mixture models is
that due to their particular structure, they allow for finding
clusters in the data or classifying a new observation.

As getting data of almost any kind has become nowadays an easy task,
mixture models are even more important in multidimensional
settings. In the last two decades, there has been an increasing number
of articles on multivariate mixtures with the conditional independence
assumption. Understanding such a model is easiest when the mixture has
a finite number of components. In this case, the model stipulates that
a population can be divided into a finite number of distinct
components, and that each multivariate observation has independent
measurements conditionally on the component to which an individual
from the population belongs.  This concept has been introduced by
\cite{HallPeter2003NEoC}, who already established some basic
identifiability results.  In that paper, the authors considered a
multivariate mixture model for results of medical tests with two
components, each of which corresponds to either a healthy or diseased
patient. Conditionally on the disease status, the medical tests are
assumed to be independent, an assumption seems to be natural in the
context of a medical study.  In the context of multinomial
classification, also called \textit{local independence}, it is stated
in \cite[Chapter 4]{rafterymbcbook} that conditional independence,
also called offers a simple way to deal with the issue of having to
estimate a large number of parameters which rapidly grows with the
dimension. Hence, conditional independence yields a parsimonious
mixture model, a which is undoubtedly a desirable feature when
considering the numerical aspects of the estimation problem.

Other research works related to the one presented here include
\cite{ElmoreR2005Aaoc}, \cite{HallPeter2005Niim},
\cite{AllmanElizabethS.2009IOPI}, \cite{ChauveauDidier2015Sefc} and
\cite{ChauveauDidier2016Nmmw}. The main accordance of all these works
is that their model is ``non-parametric" in the sense that the
component densities are completely unspecified, while the number of
components is known a priori. Hence, though dealing with a similar
subject, their model is wholly different to the best-known
``non-parametric" mixture model introduced by \cite{Lindsay}. In the
latter, the component densities are assumed to come from a known
parametric family while the mixing distribution is totally
unspecified. In the present work, we attempt to combine the concept of
conditional independence with the classical body of non-parametric
mixtures;i.e., we will use the conditional independence structure as
in the setting of \cite{Lindsay}. We allow for the more general case
in which the number of mixture components is unknown. We even go a
step further and permit the unknown mixing distribution to be nearly
arbitrary. In contrast, we shall fix the component densities to be
discrete probability mass functions (pmfs) from the class of power
series distributions (PSDs). This class includes many well-known
distributions, such as the Poisson, Geometric or Negative Binomial
distribution. Given a particular component, we then assume that the
pmf factorizes into the product of its marginal pmfs. Hence, our model
coincides with the classical non-parametric model of \cite{Lindsay}
while including at the same time the concept of conditional
independence outlined above.

Let us now explain our setting in concrete terms.  Consider
\begin{eqnarray*}
b(\theta):= \sum_{k=0}^{\infty}b_k\theta^k,
\end{eqnarray*}
for $b_k \geq0,$ to be a power series with radius of convergence
$R$. Let $\mathcal{T}:=[0,R]$ if $b(R)<\infty$ and $\mathcal{T}=[0,R)$
if $b(R)=\infty$, and define the support set
$\mathbb{K}:=\{k:b_k>0\}$. Without loss of generality, we assume that
$\mathbb{K} = \mathbb{N}$. This is the case for all well-known PSDs
with an infinite support set, i.e., with $\card(\mathbb K) =
\infty$. Famous examples are the Poisson, Geometric and Negative
Binomial distribution (see also below). In addition, even if
$\mathbb K$ were not equal to $\mathbb N = \{0,1, 2, \ldots \}$ (the
set of non-negative integers), but still infinitely large, we could
always make it equal to $\mathbb N$ by simply re-indexing its
elements.  For a detailed justification, we refer to
\cite{FadouaHaraldYong}.  For PSDs with a finite support set, i.e.,
with $\card(\mathbb K) < \infty$, it is already known that the
non-parametric maximum likelihood estimator converges to the truth
with the fully parametric rate of $n^{-1/2}$ in the Hellinger
distance. For a formal proof of this result, we refer to Appendix.

For any $\theta \in \mathcal{T}$, we can define now the corresponding
PSD
\begin{eqnarray*}
f_\theta(k):=\frac{b_k\theta^k}{b(\theta)},
\end{eqnarray*}
for $k \in \mathbb{N}$.  To provide concrete examples, consider three
well-known PSDs.
\begin{itemize}
\item \emph{The Poisson distribution:}
  $f_{\theta}(k) = e^{-\theta} \theta^k/k!$, $\theta \in [0, \infty)$,
  with radius of convergence $R = \infty$. Here, $b_k = 1/k!$ and
  $b(\theta) = e^{\theta}$.
\item \emph{The Geometric distribution:}
  $f_{\theta}(k) = (1-\theta) \theta^k$, $\theta \in [0, 1)$, with
  radius of convergence $R=1$. Here, $b_k =1$ and
  $b(\theta) = (1-\theta)^{-1}$.
\item \emph{The Negative Binomial distribution with some given
    stopping parameter $v>0$:}
  $f_{\theta}(k) = (1-\theta)^v \binom{k+v-1}{v-1}\theta^k$,
  $\theta \in [0, 1)$, with radius of convergence $R=1$. Here,
  $b_k =\binom{k+v-1}{v-1}$ and $b(\theta) = (1-\theta)^{-v}$.
\end{itemize}
Let $\Theta:= \mathcal{T}^d \subseteq \mathbb{R}^d$, where the
dimension of the mixture $d \ge 1$ is assumed to be fixed and
known. We are interested in distributions that result from mixing
given $d$-dimensional PSDs of the same family under the conditional
independence structure. More concretely, let
$\mathbf{X}_1, \ldots, \mathbf{X}_n$ be i.i.d. random vectors taking
values in $\mathbb{R}^d$, with pmf given by
\begin{eqnarray*}
  \mathbb{P}(\mathbf{X}_1 = \mathbf{k}) =: \pi_0(\mathbf{k}) = 
  \int_\Theta \prod_{j=1}^d f_{\theta_j}(k_j) dQ_0(\pmb{\theta}) = \int_\Theta \prod_{j=1}^d f_{\theta_j}(k_j) dQ_0(\theta_1, \ldots, \theta_d)
\end{eqnarray*}
with $\mathbf{k}=(k_1,\ldots,k_d) \in \mathbb N^d$ and
$\pmb{\theta} = (\theta_1, \ldots, \theta_d)$. Thus, the components
$X_1,\ldots,X_d$ of $\pmb{X}_1$ are independent conditionally that
they belong to a certain class. Here, $Q_0$ is an unknown mixing
distribution which is supported on $\Theta$.  In the particular case
where $Q_0$ has $m \ge 1$ support points,
$\pmb{\theta}_1, \ldots, \pmb{\theta}_m$, the true mixture pmf can be
rewritten as
\begin{eqnarray*}
  \pi_0(\mathbf{k})
  = \sum_{i=1}^m p_i \prod_{j=1}^d f_{\theta_{ij}}(k_j)
  = \sum_{i=1}^m p_i \prod_{j=1}^d \frac{b_{k_j}\theta_{ij}^{k_j}}{b(\theta_{ij})},
\end{eqnarray*}
with $p_i \in (0,1)$ for $i \in \{1,\ldots,m\}$ such that
$\sum_{i=1}^m p_i = 1$, and
$(\theta_{i1}, \ldots, \theta_{id}) = \pmb{\theta}_i \in \Theta$ for
$i \in \{1,\ldots,m\}$, the support points of $Q_0$. Thus,
conditionally on the $i$-th class, the multi-dimensional pmf of the
PSD family factorizes into its marginal pmfs. However, and as
mentioned above, we shall follow the route of \cite{Lindsay} and make
very little assumptions on $Q_0$. In particular, this means that $Q_0$
is allowed to have an infinite support (which can be even an
interval). Let $\widehat{Q}_n$ denote the non-parametric maximum
likelihood estimator (MLE) of the true mixing distribution $Q_0$ based
on the sample $\mathbf X_1, \ldots, \mathbf
X_n$.

Let us write
\begin{eqnarray*}
\widehat \pi_n(\mathbf k) = \int_\Theta \prod_{j=1}^d f_{\theta_j}(k_j) d\widehat{Q}_n(\pmb \theta) = \int_\Theta \prod_{j=1}^d f_{\theta_j} (k_j)d\widehat{Q}_n(\theta_1,\ldots,\theta_d),
\end{eqnarray*}
$\mathbf k = (k_1,\ldots,k_d) \in \mathbb N^d$, the corresponding MLE
of the true mixture $\pi_0$. Existence of $\widehat{Q}_n$ and
$\widehat{\pi}_n$ can be shown using Theorem 18 of
\cite{lindsay1995}. See Appendix for a formal proof. Also, for
$\mathbf k = (k_1,\ldots,k_d)$, denote by
\begin{eqnarray*}
\overline \pi_n(\mathbf k) = \frac{1}{n} \sum_{i=1}^n \mathbb{I}_{\{\mathbf{X}_i= \mathbf k\}},
\end{eqnarray*}
the empirical estimator.  Note that one of the main reasons that the
MLE is more attractive than the empirical estimator is that it
maintains the model structure. In addition, the MLE seems to handle
much better the lack of any information beyond the largest order
statistics of the observations. In fact, one can see from the
simulation results shown in Figure \ref{simpoi2d}, \ref{simgeo2d},
\ref{simnb2d}, \ref{simpoi4d}, \ref{simgeo4d}, and \ref{simnb4d} that
MLE have clearly a much better performance than the empirical
estimator in the sense of the Hellinger, $\ell_1$- and
$\ell_2$-distances. For $d=1$, the same observation was made in
\cite{FadouaHaraldYong} . There, the authors show via simulations that
the superior performance of the MLE can be explained by the
substantial difference, in favor of the MLE, of their performances at
the tail.

Recall that for two probability measures $\pi_1$ and $\pi_2$ defined
on $\mathbb N^d$, the (squared) Hellinger distance is defined as
\begin{eqnarray*}
h^2(\pi_1,\pi_2):=\frac{1}{2} \sum_{\mathbf k \in \mathbb{N}^d} \left( \sqrt{\pi_1(\mathbf k)}-\sqrt{\pi_2(\mathbf k)} \right)^2   = 1 - \sum_{\mathbf k \in \mathbb N^d} \sqrt{\pi_1(\mathbf k)\pi_2(\mathbf k)}.
\end{eqnarray*}
This paper builds on earlier work for the one-dimensional case
$d=1$. \cite{Patilea} showed that for a wide range of PSDs, the rate
of convergence of the MLE in the sense of the Hellinger distance is
$(\log n)^{1+\epsilon}n^{-1/2}$, for any $\epsilon > 0$. To obtain
this result, however, the mixing distribution was required to be
compactly supported on an interval $[0,M]$, with $0 < M < 1 \le
R$. \cite{FadouaHaraldYong} showed that for univariate mixtures of
nearly all well-known PSDs, the MLE converges to the truth at the rate
$(\log n)^{3/2} n^{-1/2}$ in the Hellinger distance. This result was
achieved under very mild assumptions, from which the most important
one is that the mixing distribution has compact support. In contrast
to \cite{Patilea}, the upper end of the support was allowed to be
arbitrary.

One of the main goals of the present work is to show that the
Hellinger distance between the true mixture $\pi_0$ and the
corresponding MLE $\widehat \pi_n$ satisfies that
\begin{eqnarray*}
h(\widehat \pi_n, \pi_0)   = O_{\mathbb P} \left( \frac{\log (nd)^{1+d/2}}{\sqrt n}\right),
\end{eqnarray*}
where $n$ and $d$ denote again the size of the observed sample and the
dimension of the multivariate mixture respectively. Note that the
aforementioned result by \cite{FadouaHaraldYong} can be recovered when
$d=1$. Furthermore, the dimension $d$ is allowed to grow with $n$. See
Remark \ref{dimdependsn} for more details.
 
\subsection{Organization of the paper}

The manuscript will be structured as follows.  The key theoretical
part of this paper is Section \ref{mlesection} where we show that for
multivariate mixtures of nearly all well-known PSDs, and under
conditional independence, the MLE converges in the Hellinger distance
at a nearly parametric rate. Herewith we mean that the parametric rate
is inflated by a logarithmic term which depends on the sample size and
the dimension of the mixture. The proof relies on techniques from
empirical process theory. While our approach resembles that of
\cite{Patilea} and \cite{FadouaHaraldYong}, this work is, to the best
of our knowledge, the first one which derives a nearly parametric rate
for multi-dimensional mixtures of PSDs with an infinite support set.

Although the convergence rate of the MLE is really fast, we believe
that it could still be improved and made fully parametric in
$\ell_p$-distances for $p \in [1, \infty]$. Unfortunately, a proof of
this stronger rate seems to be very hard to construct, even for $d=1$.
For this reason, we consider a new non-parametric estimator in Section
\ref{hybridsection} which combines the MLE and the empirical estimator
in a way that exploits the advantages of each. This \textit{hybrid}
estimator is shown to converge with the fully parametric rate of
$n^{-1/2}$ in any $\ell_p$-distance, for $p \in [1, \infty]$.  In
Section \ref{simulationsection} we present simulation results for
different multivariate PSDs, thereby supporting our theoretical
results. The same section also provides a practical application of our
findings for the famous V\'elib dataset which contains data from the
bike sharing system of Paris.  In Section \ref{testsection} we
introduce a testing procedure based on bootstrap which can be applied
to decide whether conditional independence is valid or not for a given
dataset. The practical usefulness of this test is then shown for
several multivariate PSDs with varying levels of
dependence. Furthermore, we use this testing procedure to investigate
whether the V\'elib dataset may be regarded as conditionally
independent.  We conclude this manuscript by an outlook for future
research.

In the main paper, we only present the most important proofs. The
remaining proofs, especially those which are similar to the ones given
in \cite{FadouaHaraldYong} for $d=1$ are deferred to Appendix.

\section{Rate of convergence} \label{mlesection}

\subsection{Assumptions on the mixture model}

Consider a family of PSDs
\begin{eqnarray*}
f_\theta(k) = \frac{b_k\theta^k}{b(\theta)}, k \in \mathbb N,
\end{eqnarray*}
for $\theta \in \mathcal{T}$, with $\mathcal{T}= [0, R]$ if
$b(R) < \infty$ or $\mathcal{T}=[0, R)$ if $b(R)=\infty$. Set
$\Theta:=\mathcal{T}^d$, where $d$ denotes the dimension of the
mixture.  Our goal is to estimate a multivariate mixture, where
conditionally on any mixture class, the corresponding $d$-dimensional
PSD pmf factorizes into the product of its $d$ marginal pmf's. Then,
the multivariate PSD mixture bears the form
\begin{eqnarray}\label{Model}
  \pi_0(\mathbf k)  =  \int_\Theta \prod_{j=1}^d f_{\theta_j}(k_j) dQ_0(\pmb \theta) =  \int_\Theta  \prod_{j=1}^d f_{\theta_j}(k_j)dQ_0(\theta_1,\ldots,\theta_d)
\end{eqnarray}
with $\mathbf k=(k_1,\ldots,k_d) \in \mathbb N^d$ and $Q_0$ the
unknown true mixing distribution. We estimate $\pi_0$ using
non-parametric maximum likelihood estimation based on $n$
i.i.d. $\mathbb{R}^d$-valued observations
$\mathbf{X}_1, \ldots, \mathbf{X}_n \sim \pi_0$.  In the following, we
derive in the Hellinger distance a global rate of convergence of the
MLE to the true pmf of the mixture. Note that the focus in this paper
is on estimating the mixed pmf and not the mixing distribution. We
refer the reader to Remark \ref{remMixingDis} for more comments on
this important aspect.  To derive the convergence rate of the MLE, we
need to make the following four assumptions.

\medskip

\par \noindent \textbf{Assumption (A1).}
\begin{itemize}
\item If $R < \infty$, then there exists $q_0 \in (0,1)$ such that the
  support of the true mixing distribution satisfies
  $\supp{Q_0} \subseteq [0,q_0R]^d$.
\item If $R = \infty$, then there exists $M > 0$ such that
  $\supp{Q_0} \subseteq [0,M]^d$.
\end{itemize}

\par \noindent \textbf{Assumption (A2).}
\begin{itemize}
\item If $Q_0(\{0,\ldots,0\}) > 0$, then there exists
  $\eta_0 \in (0,1)$ and $\delta_0 \in (0, R)$ small such that \\
  $Q_0(\{0,\ldots,0\}) \le 1-\eta_0$ and
  $\supp{Q_0} \cap \left \{ \cup_{j=1}^d \Big \{ \pmb{\theta}:
    \theta_j \in (0, \delta_0) \Big \} \right \} = \varnothing$.
\item If $Q_0(\{0,\ldots,0\}) = 0$, then there exists
  $\delta_0 \in (0, R)$ small such that \\
  $\supp{Q_0} \cap \left \{ \cup_{j=1}^d \Big \{ \pmb{\theta}:
    \theta_j \in (0, \delta_0) \Big \} \right \} = \varnothing$.
\end{itemize}
\par \noindent \textbf{Assumption (A3).}   There exists $V \in
\mathbb{N}$ such that $b_k/b_0 \ge k^{-k}$ \ for all $k \ge V$.  \\
\par \noindent \textbf{Assumption (A4).}  The limit
$\lim_{k \to \infty}  b_{k+1}/b_k $ exists and belongs to $[0, \infty)$. \\

\noindent In the following we comment of these four assumptions and
explain why they are reasonable.  Assumptions (A3) and (A4) are
satisfied by many well-known PSDs, including the Poisson, Geometric
and Negative Binomial and logarithmic distributions, to name only a
few. Note that Assumption (A4) implies that
\begin{eqnarray} \label{Roc} \lim_{k \to \infty} \frac{b_{k+1}}{b_k} =
  \frac{1}{R}, \ \textrm{if $R < \infty$}, \ \textrm{and} \ \lim_{k
    \to \infty} \frac{b_{k+1}}{b_k} =0, \ \textrm{if $R=\infty$}.
\end{eqnarray}
Assumption (A2) impedes the mixture from putting too much mass on the
zero vector or having support points that are very close to it. This
is again intuitive because otherwise, we would deal with nearly a
Dirac measure at zero, which is not very sensible in practice.  On the
other hand, Assumption (A1) hinders the mixture from having mass very
near the radius of convergence of the underlying PSD family. It is
clear anyway that the mixing distribution has no support beyond the
radius of convergence $R$. In fact, if this occurs, then the mixture
would not be well-defined. For the case that the radius of convergence
is infinite, this assumption states that the support of the mixing
distribution is compactly supported. Thus, Assumption (A1) is the main
assumption in this work, aside from the conditional independence
structure.  It is very important to note that none of the constants
involved in Assumptions (A1) and (A2) is supposed to be known. This
means that we are actually in the fully non-parametric setting of
\cite{Lindsay}. However, and it is to be expected, the quality of
convergence of the MLE will depend on them. This dependence is made
explicit in Theorem \ref{RateOfConvergence}.

\subsection{Rate of convergence of the non-parametric  MLE}

Throughout this section, we suppose that we are dealing with a
$d$-dimensional mixture $\pi_0$, with the conditional independence
structure, and also that Assumptions (A1) to (A4) hold true.  Let
$\widehat{\pi}_n$ denote again the non-parametric MLE of $\pi_0$ based
on $\mathbb{R}^d$-valued random vectors
$\mathbf{X}_1, \ldots, \mathbf{X}_n \stackrel{i.i.d.}{\sim}
\pi_0$. Existence of the MLE follows from Theorem 18 in Chapter 5 of
\cite{lindsay1995}. A detailed proof of existence and uniqueness of
the MLE can be found in Appendix.  In the sequel, we only deal with
the case $\mathbb K = \mathbb N$. When $\mathbb K$ is finite, the MLE
can be shown to converge to $\pi_0$ at the $n^{-1/2}-$rate, see
Appendix for a proof.

In the sequel, we will need the following quantities:
\begin{eqnarray}\label{t0tildetheta}
  t_0 =  \frac{q_0+1}{2} \mathds{1}_{\{R < \infty\}}  +  \frac{1}{2} \mathds{1}_{\{R = \infty\}}, \  \   \tilde \theta =  (q_0 R) \mathds{1}_{\{R < \infty\}} + M \mathds{1}_{\{R = \infty\}},
\end{eqnarray}
\begin{eqnarray}\label{UW}
  U =  \Big \lfloor \tilde \theta \sup_{\theta \in (0, \tilde \theta)} \frac{b'(\theta)}{b(\theta)} \Big \rfloor +1,  \   \ W =  \min \left \{ w \ge 3:  \max_{k \ge w} \frac{b_{k+1}}{b_k}  \le \frac{t_0}{\tilde \theta}   \right \},
\end{eqnarray}
and
\begin{eqnarray}\label{N0}
  && N(d, t_0, \tilde \theta, \delta_0, \eta_0) \notag \\
  && =   \bigg \lfloor \frac{1}{d} \cdot  \exp\left \{\log\left(\frac{1}{\sqrt t_0}\right) \cdot \left(U \vee V \vee W \vee \frac{b(\delta_0)}{b_0 \eta_0^{1/d}} \vee \frac{1}{\delta_0^{1/d}}   \right)    \right \} \vee \frac{1}{t_0^{W-1}(1-t_0)} \bigg \rfloor +1, \notag \\
  &&
\end{eqnarray} 
where $\lfloor z \rfloor$ denotes the integer part of some real number
$z$.

\medskip

\begin{theorem}
  \label{RateOfConvergence}
  Let $L > 2$, and let $t_0 \in (0,1)$ be the same constant as defined
  in (\ref{t0tildetheta}). Under Assumptions (A1) to (A4), there
  exists a universal constant $C > 0$ such that
\begin{eqnarray*}
  P\left(h(\widehat \pi_n, \pi_0)   >  L \frac{\log (nd)^{1+d/2}}{\sqrt n}\right) 
  & \le  &  \frac{1}{(L^2/2 - 2)^2 (\log(nd)^{2+d}}  \\
  && \ \ + \  \frac{C}{L} \frac{d 3^{d}}{\log (1/t_0)^{1+d/2}} \Big( 1 + \frac{1}{\log (1/t_0)^{1+d/2}} \Big),
\end{eqnarray*}
provided that $n \ge N(d, t_0, \tilde \theta, \delta_0, \eta_0)$,
where $N(d, t_0, \tilde \theta, \delta_0, \eta_0)$ is the same integer
in (\ref{N0}). In particular, we have that
\begin{eqnarray*}
  h(\widehat \pi_n, \pi_0)   = O_{\mathbb P} \left( \frac{\log (nd)^{1+d/2}}{\sqrt n}\right).
\end{eqnarray*}
\end{theorem}
Since the Hellinger distance dominates all $\ell_p$-distances, for
$p \in [1,\infty]$, the same rate of convergence also holds true in
all $\ell_p$-distances. However, our simulation results suggest that
MLE is $n^{-1/2}-$consistent. In fact, it is clear from the results of
Section \ref{sec:real-data} that the MLE has better performance than
the empirical and hybrid estimators, which are both known to converge
to $\pi_0$ at the parametric rate in the $\ell_1$ distance (and hence
in all $\ell_p$ distances for $p \in [1, \infty]$).

\medskip

\begin{rem}\label{dimdependsn}
  For the sake of clarity, we have assumed in Theorem
  \ref{RateOfConvergence} that the dimension $d$ is not a function of
  $n$. However, and as we will now explain, $d$ may be allowed to grow
  with $n$. If we write $L = d \ 3^{d} K$ for some constant $K > 0$,
  then it follows from Theorem \ref{RateOfConvergence} that for all
  $n \ge N(d, t_0, \tilde \theta, \delta_0, \eta_0)$
\begin{eqnarray*}
  P\left(h(\widehat \pi_n, \pi_0)   >  K  \frac{d 3^d \log (nd)^{1+d/2}}{\sqrt n}\right) 
  & \le  &  \frac{1}{(d^2 9^d K^2/2 - 2)^2 (\log(nd)^{2+d}}  \\
  && \ \ + \  \frac{C}{K} \frac{1}{\log (1/t_0)^{1+d/2}} \Big( 1 + \frac{1}{\log (1/t_0)^{1+d/2}} \Big).   
\end{eqnarray*}
Let $d = d(n)$ be increasing in $n$. First note that if we assume
without loss of generality that $\delta_0 < 1$, then combining this
with the fact that $d \ge 1$ and $ \eta_0 \in (0,1)$ implies
\begin{eqnarray*}
   N(d, t_0, \tilde \theta, \delta_0, \eta_0) & \le & N(1, t_0, \tilde
    \theta, \delta_0, \eta_0)  \\
  & = & \bigg \lfloor  \exp\left \{\log\left(\frac{1}{\sqrt
    t_0}\right) \cdot \left(U \vee V \vee W \vee
    \frac{b(\delta_0)}{b_0 \eta_0} \vee \frac{1}{\delta_0}   \right)
    \right \} \vee \frac{1}{t_0^{W-1}(1-t_0)} \bigg \rfloor +1. 
\end{eqnarray*}
This means that a convergence result can be stated for all
$ n \ge N(1, t_0, \tilde \theta, \delta_0, \eta_0)$.  Second, and in
order for the MLE to still converge to $\pi_0$ in the Hellinger
distance, we must have that
$$
\lim_{n \to \infty} \frac{d 3^d \log (nd)^{1+d/2}}{\sqrt n} = 0 \ \
\Longleftrightarrow \ \ \lim_{n \to \infty} \Big \{ \log d + d \log 3
+ (1+ d/2) \log(\log(nd)) - \log(n)/2 \Big \} = - \infty.
$$
This implies in particular that $d$ must satisfy the inequality
$d \log(3) < \log(n)/2 $. Hence, the largest dimensions that would
yield a meaningful scenario are of the form $d = d(n)=\lambda \log n $
with $0 < \lambda < 0.5/\log(3) \approx 1.047.$ In this case, we can
show after some algebra that the convergence rate is of order
\begin{eqnarray*}
\frac{ \log n \ \big(9 \log (\lambda n \log n)\big)^{\lambda \log n/2}}{\sqrt n}.
\end{eqnarray*}

\end{rem}

\begin{rem}
  \label{remMixingDis}
  This paper focuses on estimating of the mixed pmf and showing that
  it is possible to construct estimators, other than the empirical
  one, that are either a nearly and exactly
  $n^{-1/2}$-consistent. Note that this is rather a remarkable result
  given the non-parametric nature of the problem under study. In this
  sense, we do not consider in detail the \lq\lq inverse\rq\rq \
  problem of estimating the mixing distribution, which we truly
  believe deserves another paper on its own. We refer the reader to
  \cite{chen1995}, \cite{loh1996global} and \cite{heng97} where
  minimax rates were established, and which show that the rate of
  convergence can be very slow (for example of order
  $\frac{1}{(\log n)^\alpha}, \alpha > 0$. This is the case for
  example for one-dimensional mixtures of Negative Binomials with a
  smooth mixing distribution (admitting a density with respect to
  Lebesgue measure), see \cite{loh1996global}. In \cite{chen1995}, it
  was proved that for finitely supported mixing distributions (with
  unknown number of components) it is not possible to beat the rate
  $n^{-1/4}$.

  In the current work, we expect the convergence rate of the MLE
  $\widehat Q_n$ to be very slow mixture problem. However, deriving
  bounds for such a rate is far from being an easy task as it might
  require very sophisticated techniques that are specific to the PSD
  family being mixed. Nevertheless, even if it is not possible to
  investigate this aspect here, one can still think about the question
  of whether the mixing distribution in our model is identifiable. We
  answer this question positively and refer the reader to Proposition
  \ref{Identifiability} and its proof in Appendix. Note that
  identifiability is the first requirement to be checked before
  investigating consistency.

\end{rem}

\bigskip

\noindent To prove of Theorem \ref{RateOfConvergence}, we need several
auxiliary results. We start with the following lemma. Note that 1 and
2 of this lemma ( lemma \ref{Lemma1}) are properties of the PSD family
only and do not involve the dimension $d$ of the data. For this
reason, they are exactly the same as properties 1 and 2 of Lemma 2.3
in \cite{FadouaHaraldYong}. Although we refer the interested reader to
that paper for a proof, we still would like to provide some hints for
completeness. Proving property 1 uses essentially continuous
differentiability of the function $\theta \mapsto f_\theta(k)$ for any
fixed $k$. A simple calculation shows that the first derivative
$\partial f_\theta(k)/\partial \theta > 0$ for all $k \ge U$, where
$U$ is the same given in (\ref{UW}). Property 2 relies on
(\ref{Roc}). If $R < \infty$, then we know that there exists an
integer $W \ge 1$ such that for all $k \ge W$
\begin{eqnarray*}
\frac{b_{k+1}}{b_k}  \le \frac{1+ \epsilon}{R}  
\end{eqnarray*}
for a given $\epsilon > 0$. If we take $\epsilon = (1/q_0 - 1)/2$,
where $q_0 \in (0,1)$ is the same constant of Assumption (A1), then we
find that
$$1+ \epsilon = \frac{q_0 +1}{2 q_0 R} = R\frac{t_0}{\tilde \theta}$$
where $t_0$ and $\tilde \theta$ are the same constants in
(\ref{t0tildetheta}).  Imposing that $W \ge 3$ is made for convenience
as we will explain below in the proof (see Appendix). Items 3 and 4 of
Lemma \ref{Lemma1} use also properties of the PSD family but depends
on $d$, and a proof thereof is given below.

\medskip

\begin{lemma} \label{Lemma1} Let $t_0, \tilde \theta, U$ and $W$ be
  the same constants defined in (\ref{t0tildetheta}) and (\ref{UW}).
  Then, the following properties hold:
\begin{enumerate}
\item For all $k \ge U$, the map $\theta \mapsto f_\theta(k)$ is
  non-decreasing on $[0, \tilde \theta]$.

\item For all $k \ge W$, we have
\begin{eqnarray*}
  b_{k+1}  \le t_0  \frac{b_k}{\tilde{\theta}}.
\end{eqnarray*}

\item For all $K \ge \max(U, W)$, we have that
\begin{eqnarray}
  \sum_{\mathbf k: \max_{1 \le j \le } k_j \ge K+1}  \pi_0(\mathbf k)  \le A \ d \  t_0^K,
\end{eqnarray}
where 
\begin{eqnarray*}
  A := \frac{ b_W \tilde{\theta}^W}{(1-t_0) t^{W-1}_0 b(\tilde{\theta})} =  \frac{f_{\tilde \theta}(W)}{(1-t_0) t^{W-1}_0}.
\end{eqnarray*}

\item For all $k \ge W$, the map
  $k \mapsto \pi_0(k_1,\ldots,k,\ldots,k_d)$ is strictly decreasing.

\end{enumerate}
\end{lemma}

\medskip

\begin{proof}
  We start with the proof of property 3. Let $K \ge \max(U,
  W)$. First, note that
$$
\bigg \{ \mathbf k: \max_{1 \le j \le d} k_j \ge K+1 \bigg \} \subset \cup_{ 1 \le i \le d} \bigg \{ \mathbf k:  k_i \ge K+1  \bigg \}.
$$
Thus, we obtain that
\begin{eqnarray*}
  \sum_{\mathbf k: \max_{ 1 \le j \le d} k_j \ge K+1}  \pi_0(\mathbf k)  & = &   \sum_{\mathbf k: \max_{ 1 \le j \le d} k_j \ge K+1}  \int_\Theta \prod_{j=1}^d f_{\theta_j}(k_j) dQ_0(\theta_1,\ldots,\theta_d) \\
                                                                         & \le & \sum_{i=1}^d \sum_{\mathbf k:  k_i \ge K+1}  \int_\Theta \prod_{j=1}^d f_{\theta_j}(k_j) dQ_0(\theta_1,\ldots,\theta_d) \\
                                                                         & = & \sum_{i=1}^d \sum_{\mathbf k:  k_i \ge K+1}  \int_\Theta f_{\theta_i}(k_i) \prod_{j \ne i} f_{\theta_j}(k_j) dQ_0(\theta_1,\ldots,\theta_d) \\
                                                                         & \le & \sum_{i=1}^d \sum_{\mathbf k:  k_i \ge K+1} f_{\widetilde \theta}(k_i) \int_\Theta \prod_{j \ne i} f_{\theta_j}(k_j) dQ_0(\theta_1,\ldots,\theta_d) \: , \\ & \: & \textrm{using that $K \ge U$ and property 1 of Lemma \ref{Lemma1}}\\
                                                                         & = & \sum_{i=1}^d \sum_{k_i: k_i \ge K+1} \sum_{k_j \in \mathbb N: j \ne i} f_{\widetilde \theta}(k_i) \int_\Theta \prod_{j \ne i} f_{\theta_j}(k_j) dQ_0(\theta_1,\ldots,\theta_d) \\ 
                                                                         & = &   \sum_{i=1}^d \sum_{k_i: k_i \ge K+1}  f_{\widetilde \theta}(k_i) \int_\Theta \sum_{k_j \in \mathbb N: j \ne i} \prod_{j \ne i} f_{\theta_j}(k_j) dQ_0(\theta_1,\ldots,\theta_d)  \\
                                                                         & = & \sum_{i=1}^d \sum_{ k:  k \ge K+1} f_{\widetilde \theta}(k) \int_\Theta \sum_{\mathbf{l}_i} \prod_{j \ne i} f_{\theta_j}(l_j) dQ_0(\theta_1,\ldots,\theta_d)
\end{eqnarray*}
where $\mathbf l_i = (l_j)_{j \ne i} \in \mathbb N^{d-1}$.  Now note
that for each $i \in \{1,\ldots,d\}$,
$\prod_{j \ne i} f_{\theta_j}(l_j)$ is the probability that a
$(d-1)$-dimensional PSD with independent components takes on the $d-1$
values $l_j, 1\le j \le d: j \ne i$. Hence, by summing over all
possible $\mathbf l_i \in \mathbb{N}^{d-1}$, we obtain exactly
$1$. Using the fact that $Q_0$ is a probability distribution on
$\Theta$, it follows that
\begin{eqnarray*}
  \sum_{\mathbf k: \max_{ 1 \le j \le d} k_j \ge K+1}  \pi_0(k)
  & \le & \sum_{i=1}^d \sum_{k: k \ge K+1} f_{\widetilde \theta}(k) \\
  & = &  d \sum_{k : k \ge K+1} f_{\tilde \theta}(k) \\
  & = &  d \: \frac{b_W \tilde{\theta}^W}{b(\tilde{\theta})} \sum_{k: k \ge K+1}  \frac{b_k \tilde{\theta}^{k-W}}{b_W} \\
  & = & d \: \frac{b_W \tilde{\theta}^W}{b(\tilde{\theta})} \sum_{i \ge 1}  \frac{b_{K+i} \tilde{\theta}^{K-W + i}}{b_W} \\
  & \le & d \: \frac{b_W \tilde{\theta}^W}{b(\tilde{\theta})} \sum_{i \ge 1} \left( \frac{t_0}{\tilde{\theta}}\right)^{K-W+i}  \tilde{\theta}^{K-W+i}  \: , \\ & \: & \textrm{using that $K \ge W$ and property 2 of Lemma \ref{Lemma1}}\\
  & = & d \: \frac{b_W \tilde{\theta}^W}{b(\tilde{\theta})} t^{K-W}_0 \sum_{i \ge 1} t_0^i 
        =  d \: \frac{b_W \tilde{\theta}^W}{b(\tilde{\theta})} t^{K-W}_0 \frac{t_0}{1-t_0}  =  A d t_0^K.  
\end{eqnarray*}
\medskip We will now prove property 4.  Pick an arbitrary index
$j^* \in \{1,\ldots,d\}$ while fixing all the other coordinates. Let
the integer $k$ have position $j^*$ in the vector $(k_1,\ldots,k_d)$,
and assume that $k \ge W$. Then,
\begin{eqnarray*}
  \pi_0(k_1,\ldots,k + 1,\ldots,k_d) & - & \pi_0(k_1,\ldots,k,\ldots,k_d) \\
                                     & = & \int_\Theta  \Big[\prod_{j \ne j^*} f_{\theta_j}(k_j) \Big] \times f_{\theta_{j^*}}(k+1) dQ_0(\theta_1,\ldots,\theta_d) \\ & - &
                                                                                                                                                                            \int_\Theta \Big [ \prod_{j \ne j^*} f_{\theta_j}(k_j) \Big ] \times f_{\theta_{j^*}}(k) dQ_0(\theta_1,\ldots,\theta_d)\\
                                     & = & \int_\Theta \Big[ \prod_{j \ne j^*} f_{\theta_j}(k_j) \Big] \times \Big( f_{\theta_{j^*}}(k+1) - f_{\theta_{j^*}}(k) \Big) dQ_0(\theta_1,\ldots,\theta_d)\\
                                     & = & \int_\Theta \Big [ \prod_{j \ne j^*} f_{\theta_j}(k_j) \Big] \left( \frac{b_{k+1}{\theta^{k+1}_{j^*}}}{b(\theta_{j^*})} - \frac{b_{k}{\theta^k_{j^*}}}{b(\theta_{j^*})} \right) dQ_0(\theta_1,\ldots,\theta_d)  \\
                                     & = & \int_\Theta \Big [\prod_{j \ne j^*} f_{\theta_j}(k_j) \Big] b(\theta_{j^*})^{-1} \Big(b_{k+1}\theta_{j^*}^{k+1} - b_{k}\theta_{j^*}^{k} \Big) dQ_0(\theta_1,\ldots,\theta_d) \\
                                     & = & \int_\Theta \Big [\prod_{j \ne j^*} f_{\theta_j}(k_j) \Big] \ b(\theta_{j^*})^{-1} \theta_{j^*}^{k} \Big(b_{k+1}\theta_{j^*} - b_{k} \Big) dQ_0(\theta_1,\ldots,\theta_d)\\
                                     & \le & \int_\Theta \Big [\prod_{j \ne j^*} f_{\theta_j}(k_j) \Big] \ b(\theta_{j^*})^{-1} \theta_{j^*}^{k} \Big(b_{k+1}\tilde{\theta} - b_{k} \Big) dQ_0(\theta_1,\ldots,\theta_d)\\
                                     & \le & \int_\Theta \Big [\prod_{j \ne j^*} f_{\theta_j}(k_j) \Big] \ b(\theta_{j^*})^{-1} \theta_{j^*}^{k} \Big(t_0  \frac{b_k}{\tilde{\theta}} \tilde{\theta} - b_{k} \Big) dQ_0(\theta_1,\ldots,\theta_d) \: , \\ & \: & \textrm{using that $k \ge W$ and property 2 of Lemma \ref{Lemma1}}\\
                                     & = & \int_\Theta  \Big[\prod_{j \ne j^*} f_{\theta_j}(k_j) \Big] \ b(\theta_{j^*})^{-1} \theta_{j^*}^{k} b_k \Big(t_0  -1 \Big) dQ_0(\theta_1,\ldots,\theta_d)\\
                                     & = & (t_0  -1) \ \int_\Theta \Big[ \prod_{j \ne j^*} f_{\theta_j}(k_j) \Big] \ f_{\theta_{j^*}}(k)  dQ_0(\theta_1,\ldots,\theta_d) \\
                                     & = & (t_0 - 1) \ \pi_0(k_1,\ldots,k,\ldots,k_d)  < 0 \:,
\end{eqnarray*}
from which we conclude the proof.
\end{proof}

\medskip

\noindent Define now
\begin{eqnarray}\label{Kn}
  K_n:=  \min \bigg\{K \in \mathbb N:  \sum_{\mathbf k: \max_{1 \le j= d} k_j > K} \pi_0(k)  \le \frac{\log (nd)^{2+d}}{n}\bigg \},
\end{eqnarray}
and 
\begin{eqnarray}\label{taun}
\tau_n  :=  \inf_{\substack{0 \le k_j \le K_n \\ 1 \le j \le d }}  \pi_0(\mathbf k).
\end{eqnarray}
Existence of $K_n$ in (\ref{Kn}) follows immediately from the fact
that the map
$K \mapsto \sum_{\mathbf k :\max_{1 \le j \le d} k_j > K} \pi_0(k)$ is
non-increasing. Both $K_n$ and $\tau_n$ are crucial in deriving the
convergence rate of the MLE. In fact, this rate heavily depends on how
small the true pmf $\pi_0$ is at the tail. Note that the bigger $K_n$
is, the smaller is $\tau_n$. The main difficulty in the problem
studied here is due to the non-finiteness of the support. One way to
circumvent this issue is to resort to covering the support in a
progressive manner using $K_n$, which is increasing in $n$. Both $K_n$
and $\tau_n$ will play a major role in upper bounding the bracketing
entropy of a class of functions that is closely related to the set of
mixtures under study. More specifically, the related class is
$\mathcal G_n(\delta)$ defined in (\ref{Gndelta}). Upon the request of
a referee, we would like to already note here the importance of the
class $\mathcal G_n(\delta)$ in proving Theorem
\ref{RateOfConvergence}. It can be shown that
\begin{eqnarray}
  \label{BasicIneq}
  h^2(\widehat{\pi}_n, \pi_0)  \le \int \frac{\widehat{\pi}_n - \pi_0}{\widehat \pi_n + \pi_0} d(\mathbb P_n - \mathbb P).   
\end{eqnarray}
The inequality (\ref{BasicIneq}) is due to \cite{sara} and better
known under the term of the \lq\lq basic inequality\rq\rq. For a proof
we refer to \cite[Lemma 4.5]{sara}. Note that this inequality applies
in our setting since any class of mixtures is convex.  This basic
inequality enables us to relate the convergence rate of the Hellinger
distance between the MLE and $\pi_0$ to that of the empirical process
indexed by $(\pi- \pi_0)/(\pi + \pi_0)$, where $\pi$ is an element in
the mixture class.  Now, and as already mentioned above, the main
problem is that the support of the mixtures under study is infinite,
which means that both $\pi$ and $\pi_0$ decrease to $0$ in the
tail. This hinders working directly with $(\pi- \pi_0)/(\pi +
\pi_0)$. Instead, the support is \lq\lq truncated\rq\rq \ at $K_n$ in
all the $d$ components, and $(\pi- \pi_0)/(\pi + \pi_0)$ is then
decomposed into the sum of
$(\pi- \pi_0)/(\pi + \pi_0) \mathbb{I}_{ \{\mathbf k: \pi_0(\mathbf k)
  < \tau_n \}} $ and
$ (\pi- \pi_0)/(\pi + \pi_0) \mathbb{I}_{ \{\mathbf k: \pi_0(\mathbf
  k) \ge \tau_n \}}$. The first term is the most \lq\lq
troublesome\rq\rq \ since $\mathbf k$ belongs to a set where $\pi_0$
is allowed to be arbitrarily small. However, it is possible to bound
the corresponding empirical process using simple inequalities without
appealing to sophisticated techniques. The second term is \lq\lq
nicer\rq\rq \ since we know that $\pi_0 \ge \tau_n$ but requires the
use of empirical process theory. In particular, we shall need the fact
that the $\nu$-bracketing entropy of the class $\mathcal G_n(\delta)$,
for $\nu \in (0, \delta]$, is bounded above by
$$
(K_n+1)^d \log\left(\frac{1}{\tau_n}\right) + (K_n+1)^d
\log\left(\frac{\delta}{\nu}\right);
$$
see the proof of Proposition \ref{BracketingIntegral}.  The bound
above is then integrated over the $(0, \delta]$ to obtain the
so-called bracketing integral which is used to bound the expectation
of the supremum norm of the empirical processes involved in bounding
the exceedance probability $P(h(\widehat \pi_n, \pi_0) > L
\delta)$. We refer to the proof of Theorem \ref{RateOfConvergence},
where all the details are provided.

In the next lemma, we will give an upper bound for a particular
combination of $K_n$ and $\tau_n$. The proof follows a similar route
as the proof of Lemma 2.4 in \cite{FadouaHaraldYong}, and hence the
proof is relegated to Appendix.

\medskip

\begin{lemma}
  \label{Kntaun}
  Let $N(d, t_0, \tilde \theta, \delta_0, \eta_0)$ the same as in
  (\ref{N0}). For $n \ge N(d, t_0, \tilde \theta, \delta_0, \eta_0)$
  it holds that
\begin{eqnarray*}
(K_n + 1)^d \log(1/\tau_n)  \le   \frac{ d \cdot 3^{3+d}}{\log (1/t_0)^{2+d}} \log (nd)^{2+d}.
\end{eqnarray*}
\end{lemma}

\medskip

\medskip

\noindent We now move to the key part of this manuscript, which is
about finding a good upper bound for the bracketing entropy of the
class of mixtures that we consider here. In the sequel, we use the
standard notation from empirical process theory. Denote by
$\mathbb{P}$ the true probability measure; i.e.,
$d\mathbb{P}/d\mu=\pi_0$, and by $\mathbb{P}_n$ the empirical measure;
i.e., $\mathbb{P}_n:=\frac{1}{n} \sum_{i=1}^n \delta_{\mathbb{X}_i}$,
with $\delta_{\mathbb{X}_i}, i \in \{1,\ldots,n\},$ the Dirac measures
associated with our observed $d$-dimensional sample. For $\delta > 0$,
consider the class
\begin{eqnarray}
  \label{Gndelta}
  \mathcal{G}_n (\delta) :=  \left \{ \mathbb N^d \ni \mathbf k \mapsto g(\mathbf k) = \frac{\pi(\mathbf k)  - \pi_0(\mathbf k)}{\pi(\mathbf k)  +  \pi_0(\mathbf k)} \mathbb{I}_{\{ \max_{1 \le j \le K_n} k_j \le K_n \}}: \pi \in \mathcal M \  \ \textrm{such that} \ \  h(\pi, \pi_0) \le \delta   \right\},
\end{eqnarray}
where $\mathcal M$ denotes the class of multivariate mixtures $\pi$
such that
\begin{eqnarray}\label{M}
  \pi(\mathbf k) = \pi(\mathbf k, Q) = \int_\Theta \prod_{j=1}^d f_{\theta_j}(k_j) dQ_0(\pmb \theta) =  \int_\Theta  \prod_{j=1}^d f_{\theta_j}(k_j)dQ(\theta_1,\ldots,\theta_d)
\end{eqnarray}
for some arbitrary mixing distribution $Q$ defined on $\Theta$. In the
following, we compute the ``size" of this class, which is measured by
its bracketing entropy.

For a given $\nu > 0$, denote by
$H_B(\nu,\mathcal{G}_n(\delta),\mathbb{P)}$ the $\nu$-bracketing
entropy of $\mathcal{G}_n (\delta)$ with respect to $L_2(\mathbb P)$;
i.e., the logarithm of the smallest number of pairs of functions
$(L, U)$ such that $L \le U$ and $\int (U-L)^2 d\mathbb P \le \nu^2$
which is needed to cover $\mathcal{G}_n(\delta)$. Also define the
corresponding bracketing integral
\begin{eqnarray*}
  \widetilde{J}_B(\delta, \mathcal{G}_n(\delta), \mathbb P)  :=  \int_{0}^\delta  \sqrt{1 + H_B(u, \mathcal{G}_n(\delta), \mathbb P)} du.  
\end{eqnarray*}
In the following, we shall give an upper bound for this bracketing
integral. The proof is similar to the proof of Proposition 2.5 in
\cite{FadouaHaraldYong} and can be found in Appendix. Here, we focus
on the intuition behind our approach. Each element of the class
$\mathcal{G}_n(\delta)$ must have its support in the interval
$[0, K_n]^d$. Hence, as $n$ grows, the support is recovered
increasingly in all $d$ components. In choosing $K_n$, one has to
strike a balance between having a small probability at the tail and a
small entropy for the class, which obviously go in opposite
directions.

\medskip

\begin{proposition}
  \label{BracketingIntegral}
  Let $t_0$ and $N(d, t_0, \tilde \theta, \delta_0, \eta_0)$ the same
  quantities as in (\ref{t0tildetheta}) and (\ref{N0})
  respectively. For $n \ge N(d, t_0, \tilde \theta, \delta_0, \eta_0)$
  it holds that
\begin{eqnarray*}
  \widetilde{J}_B(\delta, \mathcal{G}_n(\delta), \mathbb P) \le   \frac{3^{(5+d)/2} \sqrt d \ \log (nd)^{1+d/2} \ \delta}{\log (1/t_0)^{1+d/2}} .
\end{eqnarray*}
\end{proposition}

\medskip

\noindent Now, we are finally ready to prove Theorem
\ref{RateOfConvergence}. For this, we combine all the previous results
with the ``basic inequality", already stated in (\ref{BasicIneq}).

\medskip

\medskip

\noindent \textbf{Proof of Theorem \ref{RateOfConvergence}.}\\
Let $\mathcal M$ be the class of multivariate mixtures in (\ref{M}).
Consider the sequence $\{\delta_n\}_{n \ge 1}$ defined as
\begin{eqnarray*}
\delta_n  := \frac{\log (nd)^{1+d/2}}{\sqrt n}.
\end{eqnarray*}
Consider the event $\{ h(\widehat \pi_n, \pi_0) > L \delta_n
\}$. Using the aforementioned ``basic inequality", we know that
\begin{eqnarray*}
  \int \frac{\widehat \pi_n - \pi_0}{\widehat \pi_n - \pi_0} d(\mathbb P_n - \mathbb P) \ge h^2(\widehat \pi_n, \pi_0),
\end{eqnarray*}
which means that that $\widehat \pi_n$ belongs to the subclass
$\{\pi \in \mathcal M: h(\pi, \pi_0) > L \delta_n \}$ satisfying
\begin{eqnarray*}
  \int \frac{\pi - \pi_0}{ \pi - \pi_0} d(\mathbb P_n - \mathbb P) - h^2(\pi, \pi_0)  \ge 0.
\end{eqnarray*}    
This in turn implies that
$\sup_{\pi \in \mathcal M: h(\pi, \pi_0) > L \delta_n} \left\{\int
  \frac{\pi - \pi_0}{\pi + \pi_0} d(\mathbb P_n - \mathbb P) - h^2(
  \pi, \pi_0) \right \} \ge 0$, and hence
\begin{eqnarray*}
  && P(h(\widehat \pi_n, \pi_0)   >  L \delta_n) \\
  &&   \le  P\left( \sup_{\pi \in \mathcal M: h(\pi, \pi_0)  > L \delta_n}  \left\{\int \frac{\pi - \pi_0}{\pi + \pi_0}  d(\mathbb P_n - \mathbb P)  -  h^2( \pi, \pi_0) \right \}  \ge 0 \right) \\
  && \le P \left(   \sup_{\pi \in \mathcal M: h(\pi, \pi_0)  > L \delta_n}  \left\{\int_{\{\pi _0 < \tau_n\}} \frac{\pi - \pi_0}{\pi + \pi_0}  d(\mathbb P_n - \mathbb P)  -  \frac{1}{2} h^2( \pi, \pi_0) \right \}  \ge 0 \right)   \\
  &&   + \:  P \left(  \sup_{\pi \in \mathcal M: h(\pi, \pi_0)  > L \delta_n}  \left\{\int_{\{\pi_0 \ge \tau_n\}} \frac{\pi - \pi_0}{\pi + \pi_0}  d(\mathbb P_n - \mathbb P)  -  \frac{1}{2} h^2( \pi, \pi_0) \right \}  \ge 0  \right) \\
  && =: P_1  + P_2.
\end{eqnarray*}
In the following, we will find upper bounds for $P_1$ and $P_2$. We
have that
\begin{eqnarray*}
  \int_{\{\pi_0 < \tau_n\}} \frac{\pi - \pi_0}{\pi + \pi_0}  d(\mathbb P_n - \mathbb P) & = &     \int \mathbb{I}_{\{\pi_0 < \tau_n\}} d(\mathbb P_n - \mathbb P)  - \int \mathbb{I}_{\{\pi_0 < \tau_n\}}  \frac{2 \pi_0}{\pi_0 + \pi} d(\mathbb P_n - \mathbb P) \\
                                                                                        & = &  \int \mathbb{I}_{\{\pi_0 < \tau_n\}} d(\mathbb P_n - \mathbb P) +  2 \int \mathbb{I}_{\{\pi_0 < \tau_n\}}  \frac{2 \pi_0}{\pi_0 + \pi} d\mathbb P \\
                                                                                        &&  \ - 2 \int \mathbb{I}_{\{\pi_0 < \tau_n\}}  \frac{2 \pi_0}{\pi_0 + \pi} d\mathbb P_n.
 \end{eqnarray*}
 Using the fact that $\pi + \pi_0 \ge \pi_0$, and applying the
 definitions of $K_n$ and $\delta_n$, we get that
\begin{eqnarray*}
  \int_{\{\pi_0 < \tau_n\}} \frac{\pi - \pi_0}{\pi + \pi_0}  d(\mathbb P_n - \mathbb P)  
  & \le  &  \left \vert  \int \mathbb{I}_{\{\pi_0 < \tau_n\}}d(\mathbb P_n - \mathbb P) \right \vert + 2 \sum_{\mathbf k \in \mathbb N^d}  \pi_0(\mathbf k) \mathbb{I}_{\{\pi_0(\mathbf k) < \tau_n\}} \\
  & = & \left \vert  \int \mathbb{I}_{\{\pi_0 < \tau_n\}}d(\mathbb P_n - \mathbb P) \right \vert   +  2 \sum_{k: \max_{1 \le j  \le d} k_j > K_n, \forall j=1,\ldots,d}  \pi_0(k)\\
  & \le & \left \vert  \int \mathbb{I}_{\{\pi_0 < \tau_n\}}d(\mathbb P_n - \mathbb P) \right \vert  +2 \delta_n^2.
\end{eqnarray*}
Since 
$$\sup_{\pi \in \mathcal M} \int_{\{\pi_0 < \tau_n\}} \frac{\pi - \pi_0}{\pi + \pi_0}  d(\mathbb P_n - \mathbb P) \ge \sup_{\pi \in \mathcal M, h(\pi, \pi_0) > L \delta_n} \int_{\{\pi_0 < \tau_n\}} \frac{\pi - \pi_0}{\pi + \pi_0}  d(\mathbb P_n - \mathbb P)   \ge L^2 \delta_n^2$$ 
it follows that
\begin{eqnarray*}
  P_1  & \le & P \left(    \sup_{\pi \in \mathcal M}  \int_{\{\pi _0 < \tau_n\}} \frac{\pi - \pi_0}{\pi + \pi_0}  d(\mathbb P_n - \mathbb P)    \ge \frac{L^2}{2}  \delta^2_n  \right) \\
       & \le   & P\left(   \sqrt n \left \vert  \int \mathbb{I}_{\{\pi_0 < \tau_n\}}d(\mathbb P_n - \mathbb P) \right \vert \ge  (L^2/2 - 2) \sqrt n \delta_n^2 \right)\\
       & \le &  \frac{\sum_{k \in \mathbb N^d} \pi_0(k) \mathbb{I}_{\{\pi_0(k) < \tau_n\}}}{(L^2/2 - 2)^2 n \delta_n^4} 
               \le  \frac{\delta_n^2}{(L^2/2 - 2)^2 n \delta_n^4}  =  \frac{1}{(L^2/2 - 2)^2 n \delta_n^2}.
\end{eqnarray*}
Now, we turn to finding an upper bound for $P_2$. This will be done
using the so-called peeling device. First, note that
$h(\pi,\pi_0) \le 1$ for all $\pi \in \mathcal M$. Set
$S:= \min \{s \in \mathbb N: 2^{s+1} L \delta_n \ge 1\}$. We have that
\begin{eqnarray*}
  \{\pi: h(\pi, \pi_0) > L \delta_n \} = \bigcup_{s=0}^S     \{\pi: 2^s L \delta_n < h(\pi, \pi_0) \le 2^{s+1} L \delta_n \}.
\end{eqnarray*}
Now, for $s =0, \ldots, S$, the event 
$$
\sup_{\pi \in \mathcal M: 2^s L \delta_n < h(\pi, \pi_0) \le 2^{s+1} L
  \delta_n} \left\{\int_{\{\pi _0 < \tau_n\}} \frac{\pi - \pi_0}{\pi +
    \pi_0} d(\mathbb P_n - \mathbb P) - \frac{1}{2} h^2( \pi, \pi_0)
\right \} \ge 0
$$
implies that
$$
\sup_{\pi \in \mathcal M: 2^s L \delta_n < h(\pi, \pi_0) \le 2^{s+1} L
  \delta_n} \left\{\int_{\{\pi _0 < \tau_n\}} \frac{\pi - \pi_0}{\pi +
    \pi_0} d(\mathbb P_n - \mathbb P) \right \} \ge \frac{2^{2s} L^2
  \delta_n^2}{2}
$$
and hence
$$
\sup_{\pi \in \mathcal M: h(\pi, \pi_0) \le 2^{s+1} L \delta_n}
\left\{\int_{\{\pi _0 < \tau_n\}} \frac{\pi - \pi_0}{\pi + \pi_0}
  d(\mathbb P_n - \mathbb P) \right \} \ge \frac{2^{2s} L^2
  \delta_n^2}{2}.
$$
By Markov's inequality, we obtain that 
\begin{eqnarray*}
  P_2 & \le  &  \sum_{s=0}^S P \left( \sup_{\pi \in \mathcal M: h(\pi, \pi_0)  \le 2^{s+1} L  \delta_n} \sqrt n \left \vert \int  \mathbb{I}_{\{\pi_0  \geq \tau_n\}} \frac{\pi - \pi_0}{\pi + \pi_0}  d(\mathbb P_n - \mathbb P) \right\vert  \ge   \frac{1}{2} \sqrt n 2^{2s}  L^2  \delta_n^2    \right) \\
      & = &  \sum_{s=0}^S  P \left(  \sup_{g \in \mathcal{G}_n(2^{s+1} L \delta_n)}  \vert \mathbb G_n g \vert  \ge   \frac{1}{2} \sqrt n 2^{2s} L^2  \delta_n^2   \right)  ,
\end{eqnarray*}
using property 4 Lemma \ref{Lemma1}. Here,
$\mathbb G_n f = \sqrt n (\mathbb P_n - \mathbb P) f$ is the standard
notation for the value of the empirical process at a function $f$. By
Markov's inequality,
\begin{eqnarray*}
  P_2  &\le &  \sum_{s=0}^S  \frac{2 \mathbb E\left[\Vert \mathbb G_ n \Vert_{\mathcal{G}_n(2^{s+1} L \delta_n)}\right]}{\sqrt n 2^{2s}   L^2 \delta_n^2}, \  \ \ \textrm{with $\Vert \mathbb G_n \Vert_{\mathcal F} = \sup_{f \in \mathcal F}  \vert \mathbb G_n f \vert$}.
\end{eqnarray*}
Now note that each element of the class
$\mathcal{G}_n(2^{s+1}L \delta_n)$ is bounded from above by
$1$. Furthermore, for any $g \in \mathcal{G}_n(2^{s+1}L \delta_n)$, we
have
\begin{eqnarray*}
  \mathbb P g^2  =  \sum_{\mathbf k: \max_{1 \le j \le d} k_j \le K_n} \left(\frac{\pi(\mathbf k)  - \pi_0(\mathbf k)}{\pi(\mathbf k)  + \pi_0(\mathbf k)}\right)^2 \pi_0(\mathbf k)  \le 4 \ 2^{2s+2} L^2 \delta^2_n ,
\end{eqnarray*}
using that $h(\pi, \pi_0) \le 2^{s+1} L \delta_n$ plus inequality 4.4
from \cite{Patilea}.  Thus, we may apply Lemma 3.4.2 of
\cite{aadbook}, which implies together with Proposition
\ref{BracketingIntegral} that for some universal constant $A > 0$ and
all $n \ge N(d, t_0, \tilde \theta, \delta_0, \eta_0)$
\begin{eqnarray*}
  &&\mathbb E\left[\Vert \mathbb G_ n \Vert_{\mathcal{G}_n(2^{s+1} L \delta_n)}\right] \\
  &&  \le  A \  \widetilde{J}_B(2^{s+1}L \delta_n, \mathcal{G}_n(2^{s+1} L  \delta_n), \mathbb P)   \left(1 +  \frac{\widetilde{J}_B(2^{s+1} L \delta_n, \mathcal{G}_n(2^{s+1} L \delta_n), \mathbb P)}{2^{2s+2} L^2 \delta^2_n \sqrt n} \right) \\ 
  && =  A \ \sqrt d 2^{s+1}L \delta_n \frac{3^{(5+d)/2}}{\log (1/t_0)^{1+d/2}} \log (nd)^{1+d/2}  \times  \left(1 + \frac{\sqrt d 2^{s+1 }L \delta_n \frac{3^{(5+d)/2}}{\log (1/t_0)^{1+d/2}} \log (nd)^{1+d/2}}{2^{2s+2} L^2 \delta^2_n \sqrt n} \right) \\
  && =  A \ \sqrt d 2^{s+1}L \delta_n^2 \sqrt n \frac{3^{(5+d)/2}}{\log (1/t_0)^{1+d/2}} \left(1 + \frac{\sqrt d \frac{3^{(5+d)/2}}{\log (1/t_0)^{1+d/2}} }{2^{s+1} L } \right), \ \textrm{since $\log(nd)^{1+d/2} = \sqrt n \delta_n$}, \\
  && =  A \ \left(\sqrt d 2^{s+1}L \delta_n^2 \sqrt n \frac{3^{(5+d)/2}}{\log (1/t_0)^{1+d/2}} + d \delta_n^2 \sqrt n \frac{3^{5+d}}{\log (1/t_0)^{2+d}} \right).
\end{eqnarray*}
Put $B := 4 \cdot 3^5 A$.  Using the fact that $d\ge 1$ and
$1/L^2 < 1/(2L)$ (since $L > 2$, this now gives
\begin{eqnarray*}
  P_2 & = & 2 A \sum_{s=0}^S \left(\frac{2 \cdot 3^{(5+d)/2}\sqrt d}{L \log (1/t_0)^{1+d/2}} \frac{1}{2^{s}} + \frac{d}{3^{5+d}}{L^2 \log (1/t_0)^{2+d}}    \right)  \\
      & \le & \frac{2A 3^5 3^d d}{L} \sum_{s=0}^S \left( \frac{2}{\log (1/t_0)^{1+d/2}}  \frac{1}{2^s} +  \frac{1}{2\log (1/t_0)^{2+d} } \frac{1}{4^s} \right)\\
      & \le & \frac{2A 3^5 d 3^d}{L} \left( \frac{4}{\log (1/t_0)^{1+d/2}}  +  \frac{2}{3\log (1/t_0)^{2+d} } \right)\\
& \le & \frac{2B}{L} \frac{d 3^{d}}{\log (1/t_0)^{1+d/2}} \left( 1 + \frac{1}{\log (1/t_0)^{1+d/2}} \right).
\end{eqnarray*}
By putting everything together, we finally obtain that for all
$n \ge N(d, t_0, \tilde \theta, \delta_0, \eta_0)$
\begin{eqnarray*}
P(h(\widehat \pi_n, \pi_0)  > L \delta_n)   \le \frac{1}{(L^2/2 - 2)^2 \log(nd)^{2+d}}  +   \frac{C}{L} \frac{d \ 3^{d}}{\log (1/t_0)^{1+d/2}} \left( 1 + \frac{1}{\log (1/t_0)^{1+d/2}} \right),
\end{eqnarray*} 
where $C := 2B = 8 \cdot 3^5 A$ is a universal constant. \hfill $\Box$

\section{The hybrid estimator: a non-parametric estimator with a
  parametric rate} \label{hybridsection}

In the previous section, we have shown that for multivariate mixtures
of PSDs with the conditional independence structure, the MLE converges
to the true mixture at a rate very close to parametric. Although this
rate is really fast, our simulation results in Section
\ref{simulationsection} suggest that it can still be improved. We
conjecture that at least in the $\ell_p$-distance, the MLE should
converge with the fully parametric rate of $n^{-1/2}$. Unfortunately,
there is no proof for this stronger rate which is available at the
moment. This is why we are now taking a different route and introduce
a new non-parametric estimator which turns out to converge at the
$n^{-1/2}$-rate.

It is a well-known fact that the empirical estimator converges to the
true mixture with the fully parametric rate in any $\ell_p$-distance,
for $p \ge 2$. See for example Theorem 3.1 in \cite{hannajon} (note
that convergence in $\ell_2$ implies convergence in $\ell_p$, for
every $p \in [2, \infty]$). Proposition \ref{empirical} below shows
that in our setting, the parametric rate holds true even for
$p=1$. However, the empirical estimator suffers from the disadvantage
that it puts zero mass in the tails.  In other words, although the
empirical estimator has excellent convergence properties, it does cope
well with the lack of information beyond the largest order statistic.
To improve the behavior at tail, we construct a new estimator where we
replace the empirical estimator in the tails by the MLE. It turns out
that this hybrid estimator combines the fast convergence rate of the
empirical estimator with the nice property of the MLE that it does not
vanish in the tails. Note that the hybrid estimator is not necessarily
an element of the class of mixtures under study. Hence, we see its
value in the fact that it shows that if the MLE performs better than
both the empirical and hybrid estimators, which both are
$n^{-1/2}$-consistent in the $\ell_p$-norms for $p \in [1, \infty]$,
then the MLE must be also $n^{-1/2}$-consistent.  Furthermore, we
believe that the hybrid estimator can offer a very good starting point
for pushing the theory further to show the latter result.

In the following proposition, we will show the fast convergence rate
of the empirical estimator.

\medskip

\begin{proposition}\label{empirical}
  For $ \mathbf k=(k_1,\ldots,k_d) \in \mathbb N^d$,
  let
  $$\pi_0(\mathbf k) = \int_\Theta \prod_{j=1}^d f_{\theta_j}(k_j)
  dQ_0(\theta_1,\ldots,\theta_d)$$ as defined above, and let
  $\overline \pi_n(\mathbf k)$ denote again the empirical estimator of
  $\pi_0$ based on i.i.d. $d$-dimensional random vectors
  $\mathbf{X}_1,\ldots\mathbf{X}_n \sim \pi_0$. Then, it holds that
\begin{eqnarray*}
\sum_{\mathbf k \in \mathbb N^d} \sqrt{\pi_0(\mathbf k)} < \infty. 
\end{eqnarray*}
Moreover, for all $p \in [1, \infty]$, we have that
\begin{eqnarray*}
\ell_p(\bar \pi_n, \pi_0)   =  O_{\mathbb P}(1/\sqrt{n}).
\end{eqnarray*}

\end{proposition}

\medskip
 
\begin{proof}
  Let $U$ and $W$ be the same integers as in (\ref{UW}). A sum over
  all $\mathbf k \in \mathbb N^d$ can be decomposed into $2^d$ sums,
  depending on whether a component $k_j$ is at smaller or larger than
  $\max(U,W)$. Now, let $i \in \{0,\ldots,d\}$ be arbitrary, and let
  us consider the sum over all those $k \in \mathbb N^d$ such that $i$
  components are at most $W$, and hence $d-i$ components are larger
  than $W$. Without loss of generality, we may assume that the first
  $i$ components are at most $W$. Applying properties 1 and 2 of Lemma
  \ref{Lemma1}, we can write that
\begin{eqnarray*}
  && \sum_{\substack{k_j \le W, j=1,\ldots,i \\ k_j > W, j=i+1,\ldots,d}}
  \sqrt{\pi_0(k)}  \\ &=  & \sum_{\substack{k_j \le W, j=1,\ldots,i \\
  k_j > W, j=i+1,\ldots,d}}  \sqrt{\int_\Theta \prod_{j=1}^d
  f_{\theta_j}(k_j) dQ_0(\theta_1,\ldots,\theta_d)} \\ 
                                                                                        & = & \sum_{ k_1 \le W} \ldots \sum_{k_i \le W} \sum_{ k_{i+1} > W} \ldots \sum_{k_d > W} \sqrt{\int_\Theta \prod_{j=1}^d f_{\theta_j}(k_j) dQ_0(\theta_1,\ldots,\theta_d)} \\
                                                                                        & \le & \sum_{k_1 \le W} \ldots \sum_{k_i \le W} \sum_{ k_{i+1} > W} \ldots \sum_{k_{d-1} > W}   \sqrt{\int_\Theta \prod_{j=1}^{d-1} f_{\theta_j}(k_j) dQ_0(\theta_1,\ldots,\theta_d)} \cdot \left(\sum_{k_d \ge W} \sqrt {f_{\widetilde \theta}(k_d)}\right) \\
                                                                                        & = & \sum_{k_1 \le W} \ldots \sum_{k_i \le W} \sum_{k_{i+1} > W} \ldots \sum_{ k_{d-1} > W}   \sqrt{\int_\Theta \prod_{j=1}^{d-1} f_{\theta_j}(k_j) dQ_0(\theta_1,\ldots,\theta_d)} \cdot \left(\sum_{k \ge W} \sqrt {f_{\widetilde \theta}(k)}\right) \\
                                                                                        & \le & \sum_{ k_1 \le W} \ldots \sum_{k_i \le W} \sqrt{\int_\Theta \prod_{j=1}^{i} f_{\theta_j}(k_j) dQ_0(\theta_1,\ldots,\theta_d)}  \cdot \left(\sum_{ k \ge W} \sqrt {f_{\widetilde \theta}(k)}\right)^{d-i} \\
                                                                                        & \le &  (W+1)^{i} \left(\sum_{k \ge W} \sqrt {f_{\widetilde \theta}(k) }\right)^{d-i} =  (W+1)^i \left(\sum_{ k \ge W} \frac{\sqrt{b_k} \tilde{\theta}^{k/2}}{\sqrt{b(\tilde{\theta})}}\right)^{d-i} \\
                                                                                        & \le & (W+1)^i \left(\sum_{k \ge W} \frac{\sqrt{b_W}}{\sqrt{b(\tilde{\theta})}} \left(\frac{t_0}{\tilde{\theta}}\right)^{(k-W)/2} \tilde{\theta}^{k/2}\right)^{d-i}
  \\ & = &  C \Big(\sum_{k \in \mathbb N: k \ge W} t^{(k-W)/2}_0 \Big)^{d-i} = C \Big(\frac{1}{1-\sqrt{t_0}}\Big)^{d-i}< \infty,
\end{eqnarray*}
where $C > 0$ depends on $i$, $d$, $W$, $b_W$, $\tilde{\theta}$ and
the value $b(\tilde{\theta})$. Now, the index $i \in \{0,\ldots,d\}$
has been chosen arbitrary, meaning that the whole sum
$\sum_{k \in \mathbb N^d} \sqrt{\pi_0(k)} $ can be decomposed into
$2^d$ finite sums, and hence is finite.

For the second assertion, note that $\vert \overline \pi_n(\mathbf k) - \pi_0(\mathbf k) \vert \ge \vert \overline \pi_n(\mathbf k) - \pi_0(\mathbf k) \vert^p$ for all $p \ge 1$ and for all $\mathbf k \in \mathbb{N}^d$. Hence, it is enough to show the result for $p =1$.  Applying Fubini's theorem and Jensen's inequality, we get
\begin{eqnarray*}
  \mathbb{E}\Big[\sum_{\mathbf k \in \mathbb N^d} \vert \overline \pi_n(\mathbf k) - \pi_0(\mathbf k) \vert\Big]
  & \le &    \sum_{\mathbf k \in \mathbb N^d}  \sqrt{\mathbb{E}\Big[(\overline \pi_n(\mathbf k) - \pi_0(\mathbf k))^2\Big]}
          =  \sum_{\mathbf k \in \mathbb N^d} \sqrt{\frac{1}{n} \pi_0(\mathbf k) (1-\pi_0(\mathbf k))} \\
  & = & \frac{1}{\sqrt{n}}  \sum_{\mathbf k \in \mathbb N^d} \sqrt{\pi_0(\mathbf k) (1-\pi_0(\mathbf k))}
        \le     \frac{1}{\sqrt{n}}  \sum_{\mathbf \in \mathbb N^d} \sqrt{\pi_0(\mathbf k)}.
\end{eqnarray*}
We conclude the proof by using Markov's inequality and the first
assertion.
\end{proof}

\medskip

\noindent In the following proposition we introduce the hybrid
estimator and prove that it converges to the truth at the promised
rate of $n^{-1/2}$.

\medskip

\begin{proposition}
  \label{hybrid}
  Let $\widehat{\pi}_n$ denote again the MLE of
  $\pi_0 \in \mathcal{M}$.  Let $\widetilde{K}_n > 0$ be the smallest
  integer $K$ such that
\begin{eqnarray*}
  \sum_{\mathbf k: \max_{1 \le j \le d} k_j > K}  \widehat \pi_n(\mathbf k)   \le \frac{1}{\log (nd)^{2+d}}.
\end{eqnarray*}
Then, the hybrid estimator $\widetilde{\pi}_n$ defined as
\begin{eqnarray*}
  & \widetilde{\pi}_n(\mathbf k)  := \tilde{s}_n^{-1} \left(\overline
  \pi_n(\mathbf k) \mathbb{I}_{\{\max_{j=1,\ldots,d} k_j \le
  \widetilde{K}_n\}}   +  \widehat \pi_n
  \mathbb{I}_{\{\max_{j=1,\ldots,d} k_j > \widetilde{K}_n\}} \right),
  \\
  & \hspace{2cm} \textrm{with} \ \tilde{s}_n = \sum_{\mathbf k \in N^d}
  \left(\overline \pi_n(\mathbf k) \mathbb{I}_{\{\max_{j=1,\ldots,d}
  k_j \le \widetilde{K}_n\}}   +  \widehat \pi_n
  \mathbb{I}_{\{\max_{j=1,\ldots,d} k_j > \widetilde{K}_n\}}\right) 
\end{eqnarray*}
satisfies that
\begin{eqnarray*}
  \ell_p(\widetilde{\pi}_n, \pi_0)   =  O_{\mathbb P}(1/\sqrt{n})
\end{eqnarray*}
for all $p \in [1, \infty]$.
\end{proposition}

\begin{proof}
  It suffices to show the result for $p =1$. Assume that we proved
  this for
\begin{eqnarray*}
  \check{\pi}_n(\mathbf k)  =   \overline \pi_n(\mathbf k) \mathbb{I}_{\{\max_{j=1,\ldots,d} k_j \le \widetilde{K}_n\}}   +  \widehat \pi_n  \mathbb{I}_{\{\max_{j=1,\ldots,d} k_j > \widetilde{K}_n\}}, \ \mathbf k \in  \mathbb N^d,
\end{eqnarray*}
that is, suppose that we know that
$\sum_{\mathbf k \in \mathbb N^d} \vert \check{\pi}_n(\mathbf k) -
\pi_0(\mathbf k) \vert = O_{\mathbb P}(1/\sqrt n)$. Then,
\begin{eqnarray*}
  \vert \tilde s_n  -1 \vert  & = & \left \vert \sum_{\mathbf k \in \mathbb N^d} \check{\pi}_n(\mathbf k)  -  \sum_{\mathbf k \in \mathbb N^d} \pi_0(\mathbf k) \right \vert  \le  \sum_{\mathbf k \in \mathbb N^d} \left \vert \check{\pi}_n(\mathbf k)  - \pi_0(\mathbf k) \right \vert  
\end{eqnarray*}
and hence $\vert \tilde s_n -1 \vert = O_{\mathbb P}(1/\sqrt n)$.
This implies
\begin{eqnarray*}
  \sum_{\mathbf k \in \mathbb N^d} \vert \widetilde{\pi}_n(\mathbf k) - \pi_0(\mathbf k) \vert =  \sum_{\mathbf k \in \mathbb N^d} \left \vert \frac{1}{\tilde s_n} \check{\pi}_n(\mathbf k) - \pi_0(\mathbf k) \right \vert  & \le  &    \frac{1}{s_n} \sum_{\mathbf k \in \mathbb N^d} \vert \check{\pi}_n(\mathbf k) - \pi_0(\mathbf k) \vert +  \vert \tilde s_n -1 \vert  \\
                                                                                                                                                                                                                              & \le &  2 \sum_{\mathbf k \in \mathbb N^d} \vert \check{\pi}_n(\mathbf k) - \pi_0(\mathbf k) \vert  + \vert \tilde s_n -1 \vert = O_{\mathbb P}(1/\sqrt n)
\end{eqnarray*}
using the fact that for $n$ large enough $\tilde s_n \ge 1/2$. Now, we
will show
$\sum_{\mathbf k \in \mathbb N^d} \vert \check{\pi}_n(\mathbf k) -
\pi_0(\mathbf k) \vert = O_{\mathbb P}(1/\sqrt n)$.  We have that
\begin{eqnarray}
  \label{zerleg}
  \vert \check{\pi}_n(k)  -  \pi_0(k) \vert \le  \vert \overline \pi_n(k)  - \pi_0(k) \vert \mathbb{I}_{\{\max_{j=1,\ldots,d} k_j \le \widetilde{K}_n \}}   +   \vert \widehat{\pi}_n(k)  - \pi_0(k) \vert \mathbb{I}_{\{\max_{j=1,\ldots,d} k_j > \widetilde{K}_n\}}.
\end{eqnarray}
Using the Cauchy-Schwarz inequality, we obtain that
\begin{eqnarray*}
  && \sum_{\substack{ \mathbf k: \\ \max_{j=1,\ldots,d} k_j > \widetilde{K}_n}} \vert \widehat{\pi}_n(k)  - \pi_0(k) \vert 
  =  \sum_{\substack{\mathbf k: \\ \max_{j=1,\ldots,d} k_j > \widetilde{K}_n}} \left  \vert \sqrt{\widehat{\pi}_n(k)}  - \sqrt{\pi_0(k)}  \right \vert  \big(\sqrt{\widehat{\pi}_n(k)} +  \sqrt{\pi_0(k)} \big)  \\
  &&  \le \left \{ \sum_{\mathbf k: max_{j=1,\ldots,d} k_j > \widetilde{K}_n} \big(\sqrt{\widehat{\pi}_n(\mathbf k)}  - \sqrt{\pi_0(\mathbf k)}  \big)^{2} \right \}^{1/2} \cdot  \left \{\sum_{\mathbf k: \max_{j=1,\ldots,d} k_j > \widetilde{K}_n} \big(\sqrt{\widehat{\pi}_n(\mathbf k)} +  \sqrt{\pi_0(\mathbf k)} \big)^2\right \}^{1/2}\\
  & & \le  \sqrt 2  h(\widehat \pi_n, \pi_0)  \cdot  \sqrt 2   \left \{\sum_{\mathbf k: \max_{j=1,\ldots,d} k_j > \widetilde{K}_n} \widehat {\pi}_n(k)  + \sum_{\mathbf k: \max_{j=1,\ldots,d} k_j > \widetilde{K}_n}  \pi_0(\mathbf k)  \right\}^{1/2}\\
  & &  \le  2  h(\widehat \pi_n, \pi_0)  \cdot  \left  \{\sum_{ \mathbf k:\max_{j=1,\ldots,d} k_j > \widetilde{K}_n} \vert \widehat {\pi}_n(k)  - \pi_0(k) \vert +  2  \sum_{ \mathbf k: \max_{j=1,\ldots,d} k_j > \widetilde{K}_n}  \widehat \pi_n(k)  \right\}^{1/2}  \\
  & &  \le  O_{\mathbb P}\left(\frac{\log (nd)^{1+d/2}}{\sqrt n}\right)  \cdot  \left(O_{\mathbb P}\left(\frac{(\log (nd))^{1+d/2}}{\sqrt n}\right)   + (\log (nd))^{-(2+d)}  \right)^{1/2}  
  \\ && =  O_{\mathbb P}(1/\sqrt{n}),
\end{eqnarray*}
where we have applied Theorem \ref{RateOfConvergence}, our convergence
result for the MLE. We conclude by using Proposition \ref{empirical},
which implies that the sum in the first term of (\ref{zerleg}),
$\sum_{k \in \mathbb N^d} \vert \overline \pi_n(k) - \pi_0(k) \vert
\mathbb{I}_{\{\max_{j=1, \ldots, d} k_j \le \widetilde{K}_n \}}$, is
$ O_{\mathbb P}(1/\sqrt n)$.
\end{proof}

\medskip

\noindent As written at the beginning of this section, a disadvantage
of the empirical estimator is that it puts zero mass in the tail. This
does not happen with the hybrid estimator with probability tending to
$1$ as the sample size grows to infinity. To show this, we make use of
the following fact whose proof is relegated to the appendix.

\medskip

\begin{proposition}
  \label{Boundfortildek}
  Let $\widetilde{K}_n$ be defined as in Proposition
  \ref{hybrid}. Then, it holds that
\begin{eqnarray*}
(\widetilde{K}_n +1)^d (1- \pi_0(\widetilde{K}_n,\ldots,\widetilde{K}_n))^n = o_{\mathbb P}(1).
\end{eqnarray*}
\end{proposition}

\medskip

\noindent This then leads to the following result.

\medskip

\begin{proposition}\label{Nonzero}
We have that
\begin{eqnarray*}
  \lim_{n \to \infty}  P\left( \min_{\mathbf k \in \mathbb{N}^d: \max_{ 1 \le j \le d} k_j \le \widetilde{K}_n}   \overline \pi_n(\mathbf k)  > 0  \right)   = 1.
\end{eqnarray*}
In particular,  it holds that
\begin{eqnarray*}
\lim_{n \to \infty}  P\left( \min_{\mathbf k \in \mathbb{N}^d}   \widetilde \pi_n(k)  > 0  \right)   = 1.
\end{eqnarray*}
\end{proposition}

\medskip

\begin{proof}
  For any fixed $\mathbf k \in \mathbb N^d$, it is clear that
  $n \overline \pi_n(\mathbf k) \sim \text{Bin}(n, \pi_0(\mathbf k))$.
  Then, for $n$ large enough
  \begin{eqnarray*}
    P\left( \min_{ \mathbf k: \max_{1 \le j \le d} k_j \le \widetilde{K}_n}   \overline \pi_n(\mathbf k)  > 0  \right)    
    & \ge &  1 -   \sum_{j=1}^d \sum_{k_j =0}^{\widetilde{K}_n}
    P(\overline \pi_n(k_1, \ldots, k_d) = 0)  \\ &= &  1 -   \sum_{j=1}^d \sum_{k_j =0}^{\widetilde{K}_n}  (1- \pi_0(k_1, \ldots, k_d))^n  \\
                                                                                                        & \ge   & 1- (\widetilde{K}_n +1)^d  (1-\pi_0(\widetilde{K}_n,\ldots,\widetilde{K}_n))^n,
\end{eqnarray*}
where in the last step we applied property 4 of Lemma
\ref{Lemma1}. Proposition \ref{Boundfortildek} concludes the proof.
\end{proof}

\section{Simulations and real data application}
\label{simulationsection}
We now present results of simulations for conditionally independent
mixtures of Poisson, Geometric and Negative Binomial, for varying
dimensions. These simulations do not only support our theoretical
findings, they even suggest that the MLE must be fully parametric in
the $\ell_1$-distance (and hence in any $\ell_p$-distance, for
$p \in [1, \infty]$). The simulation results are supplemented by a
real data application for the famous V\'elib data set about the bike
sharing system of Paris.

\subsection{The algorithm}

The MLE can be computed using the algorithm described in
\cite{wang-wang-2015} and \cite{hu-wang-2021}. For self-containment,
we describe it as follows (with slight modifications). Given
observations $\mathbf k_1, \dots, \mathbf k_n \in \mathbb{K}^d$, the
log-likelihood function is given by
\begin{align*}
  \ell(Q) = \sum_{i=1}^n
  \log\left(\int \prod_{j=1}^d f_{\theta_j}(k_{ij}) d Q(\theta_1,
  \ldots, \theta_d)\right). 
\end{align*}
Since the non-parametric MLE of $Q$ must be a discrete distribution
function with no more support points than the number of distinct
observations (see \cite{laird-1978,lindsay-1983}), one only needs to
consider a discrete maximizer. Let such a discrete $Q$ have support
points $\pmb{\theta}_1, \dots, \pmb{\theta}_m \in \mathbb{R}^d$, and
denote their associated probability masses by $p_1, \dots, p_m$,
respectively. The mixture can then be rewritten as
\begin{align*}
  f_Q(\mathbf k_i) = \sum_{l=1}^m p_l f_{\pmb{\theta}_l}(\mathbf k_i) = \sum_{l=1}^m
  p_l \prod_{j=1}^d f_{\theta_{lj}}(k_{ij}).
\end{align*}
Finding the non-parametric MLE of $Q$ is equivalent to finding its
support and probability vectors
$\vartheta = (\pmb{\theta}_1, \dots, \pmb{\theta}_m)$ and
$p = (p_1, \dots, p_m)^T$, including their common length $m$. We may
also write $\ell(Q)$ equivalently as $\ell(p, \vartheta)$.  To this
aim, consider first updating $p$ with $\vartheta$ fixed. This can be
achieved using the Taylor series approximation to the log-likelihood
with respect to $p$. Since
\begin{eqnarray*}
  \frac{\partial \ell}{\partial p}
   = S^T \mathbbm{1}, \ \textrm{and} \  
  \frac{\partial^2 \ell}{\partial p \partial p^T} = -S^T S,
\end{eqnarray*}
where
$S = (\partial \ell/\partial \theta_1, \dots, \partial \ell/\partial
\theta_m)^T$ and $\mathbbm{1} = (1, \dots, 1)^T$, the quadratic
Taylor series expansion about $p$ is given by
\begin{align*}
  l(p, \vartheta) - l(p', \vartheta) 
   \approx - \mathbbm{1}^T S(p' - p)+ \frac{1}{2}(p' -
    p) S^T S(p' - p) = \frac{1}{2} \| S p' - \mathbbm{2} \|^2 - \frac{n}{2},
\end{align*} 
where $\mathbbm{2} = (2, \dots, 2)^T$. This means that maximizing
$l(p', \vartheta)$ over $p'$ in the neighborhood of $p$ can be
approximately achieved by solving the following least squares
regression problem under the positivity and unity constraints:
\begin{align}
  \label{eq:nnls1}
  \min_{p'}\|S p' - \mathbbm{2}\|, ~~~ \mathrm{subject~to~}  
  p'^T \mathbbm{1} = 1, p' \ge 0.
\end{align}
Solving \eqref{eq:nnls1}, followed by a proper line search, will
result in some mixing proportions becoming exactly equal to 0. This is
desirable for computing the non-parametric MLE since the support
points associated with mixing proportions $0$ are redundant in the
mixture representation and can be discarded immediately. This allows
the support set to shrink, if necessary.

To expand the support set in an efficient way, the gradient function
is to be used. This is defined as
\begin{align*}
  d(\pmb{\theta}; Q) = \left.\frac {\partial \ell((1-\epsilon) Q + \epsilon
  \delta_{\pmb \theta}) } {\partial \epsilon}\right|_{\epsilon= 0+} =
  \sum_{i=1}^n \frac{f_{\pmb \theta}(\mathbf k_i)} {f_Q(\mathbf k_i)} - n,
\end{align*}
where $\delta_\theta$ denotes the Dirac measure at $\theta$. The local
maxima of the gradient function are deemed good candidate support
points; see \cite{wang-2007}. In a multi-dimensional space, however,
finding each of these local maxima can be computationally challenging,
and more so here as this is required for each iteration of the
algorithm. To resolve this issue, \cite{wang-wang-2015} proposed a
strategy that uses a ``random grid'', by turning the gradient function
into a finite mixture pmf and drawing a random sample from it. To do
this, one first removes the additive constant $-n$ and then turns the
remaining sum into a finite mixture pmf of $\pmb \theta$ (not
$\mathbf k$) via normalizing the coefficients. Note that
$f_{\pmb \theta}(\mathbf k)$ is non-negative but not necessarily a pmf
for $\pmb \theta$, and thus it may need normalization as well. Because
of the different role now played by $\pmb \theta$, the resulting
distribution family may also be different. For example, the Poisson
pmf $f_{\pmb \theta}(\mathbf k)$ (in terms of $\mathbf k$) is
interestingly turned into a Gamma density (in terms of $\pmb \theta$),
and the Geometric or the Negative Binomial pmf into a Beta density.
Sampling from the resulting finite mixture is straightforward, and
using a sample size $20$ seems sufficient in practice. The rationale
behind this strategy is that more random points tend to be generated
in the area with larger gradient values, thus increasing the
possibility of not missing out the areas with a local maximum, in
particular one with the global maximum. To locate more precisely the
local maxima in the areas, we run $100$ iterations of the Modal EM
algorithm \cite{li-ray-lindsay-2007}, starting with both the randomly
generated points and the support points of the current $Q$. To save
computational cost, one does not have to use all of the resulting
points but only the best one (if there is at least one) around each
current support point. The selected points are added to the support
set of the current $Q$, with zero probability masses. The mixing
proportions of all support points are then updated by using the method
described above.  The above strategy allows the support set to expand
or shrink rapidly, at an exponential rate if necessary. This is
critically important for efficient computation, especially when the
solution contains many support points. Certain variants of the above
algorithm can also be adopted, e.g., adding a few iterations of the EM
algorithm \cite{dempster-laird-rubin-1977} that updates all the
parameter values of the finite mixture obtained after problem
\eqref{eq:nnls1} is solved.

\subsection{Simulation studies}

We now investigate numerically the asymptotic behavior of our
estimators by carrying out a simulation study using the algorithm
described above. Here, we consider conditionally independent mixtures
of three component distribution families: the Poisson, Geometric and
Negative Binomial distribution. For the dimension, we choose
$d \in \{2, 4\}$. Also, the sample size is set to
$n = 100, 1000, \ldots, 10^8$ for $d=2$ and
$n = 100, 1000, \ldots, 10^6$ for $d=4$. We also study the empirical
estimator (denoted by \emph{Empricial}), the hybrid estimator
(\emph{Hybrid}) and the non-parametric maximum likelihood estimator
(\emph{MLE}). Three performance measures scaled by $\sqrt n$ are
calculated: the Hellinger, the $\ell_1$- and the $\ell_2$-distances.

\begin{figure*}[!tbh]
  \centering
  \includegraphics[width=1\textwidth]{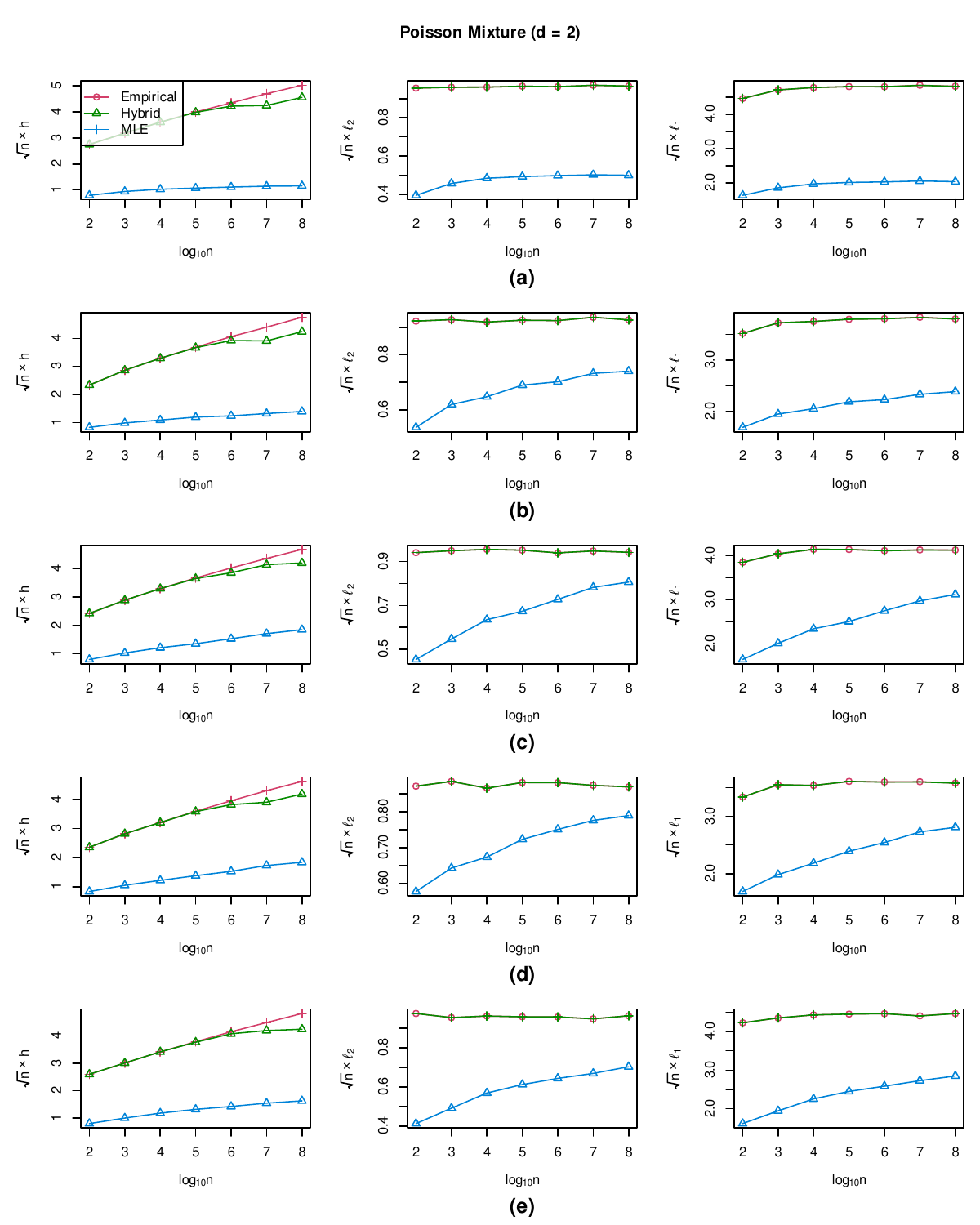}
  \caption{Two-dimensional mixtures of Poisson. In (a), the mixing
    distribution has two support points; in (b), it has four support
    points; in (c), it is uniform; in (d), it is a combination of a
    point mass and a uniform distribution; in (e), it is a combination
    of two uniform distributions.}
  \label{simpoi2d}
\end{figure*}

\begin{figure*}[!tbh]
  \centering
  \includegraphics[width=1\textwidth]{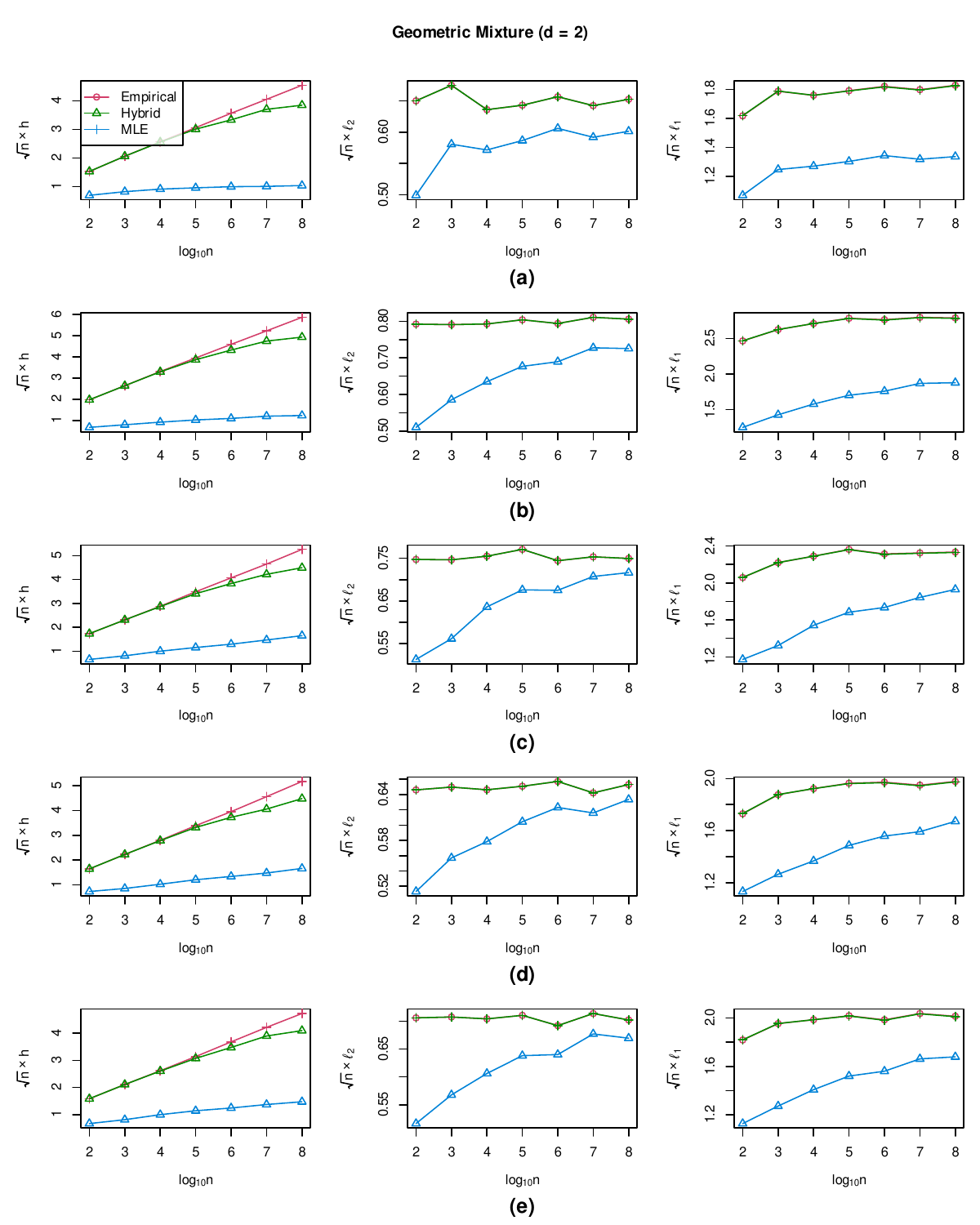}
  \caption{Two-dimensional mixtures of Geometric, for the same mixing
    distributions.} 
  \label{simgeo2d}
\end{figure*}

\begin{figure*}[!tbh]
  \centering
  \includegraphics[width=1\textwidth]{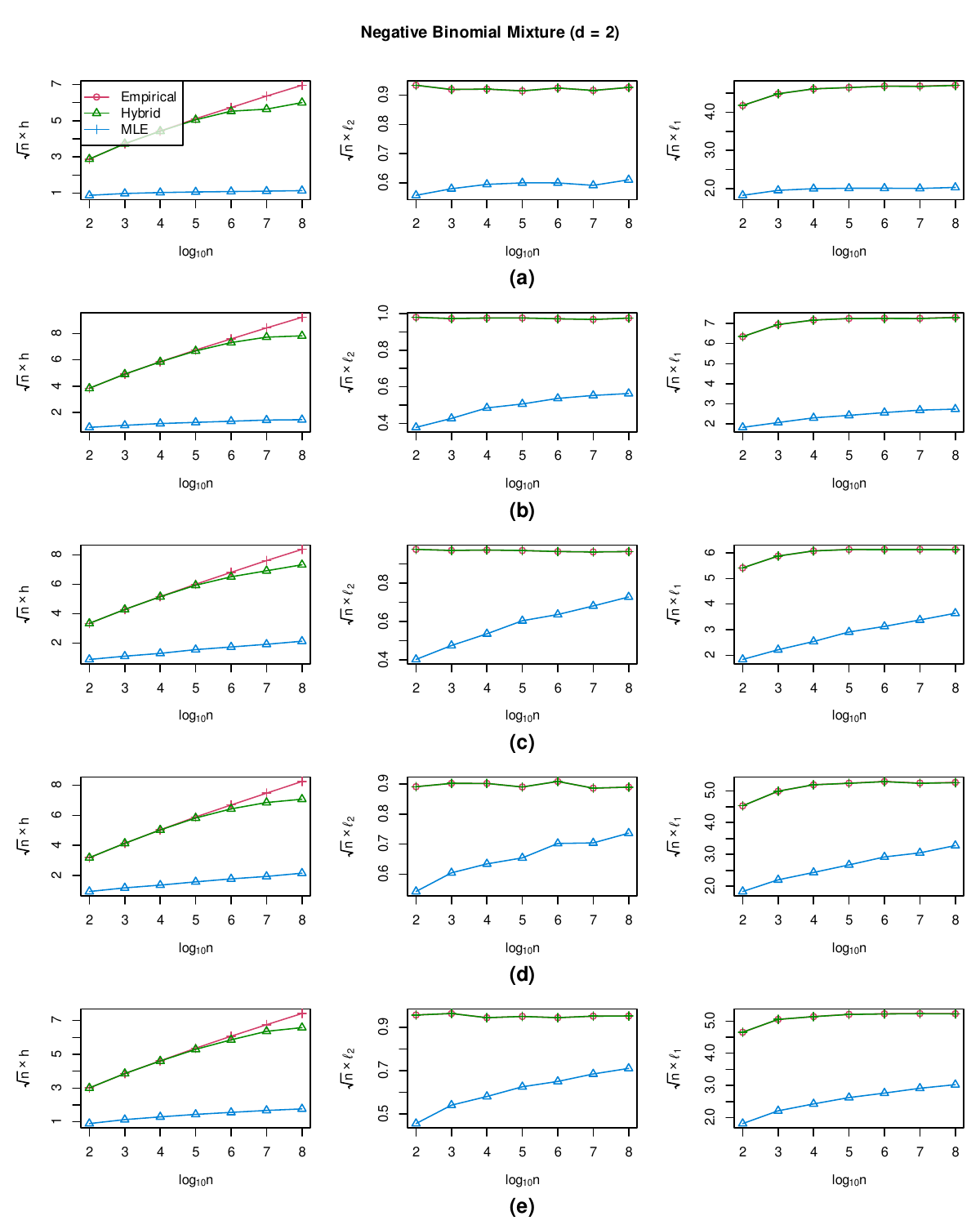}
  \caption{Two-dimensional mixtures of Negative Binomial, for the same
    mixing distributions.} 
  \label{simnb2d}
\end{figure*}

\begin{figure*}[!tbh]
  \centering
  \includegraphics[width=1\textwidth]{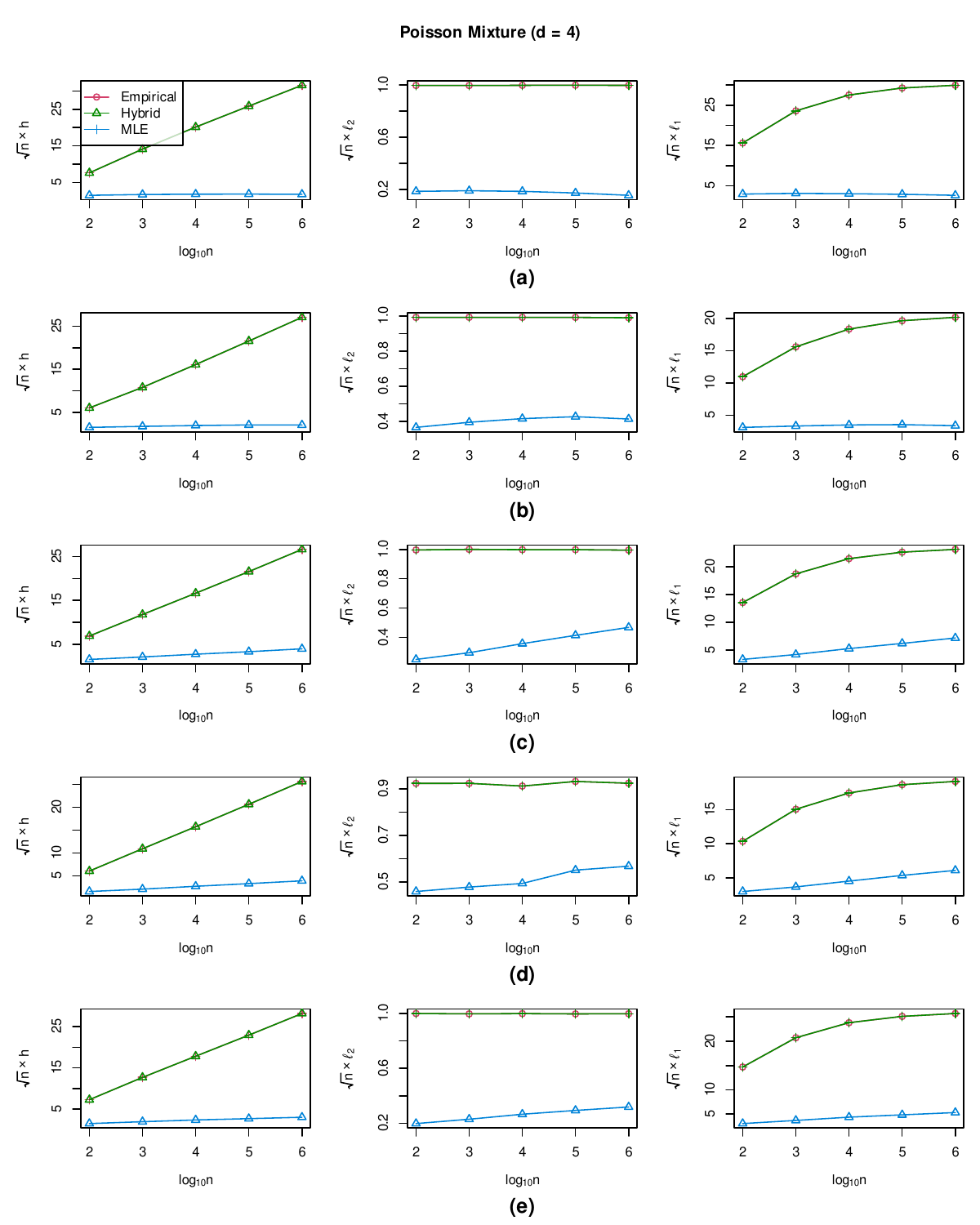}
  \caption{Four-dimensional mixtures of Poisson, for the same mixing
    distributions.} 
  \label{simpoi4d}
\end{figure*}

\begin{figure*}[!tbh]
  \centering
  \includegraphics[width=1\textwidth]{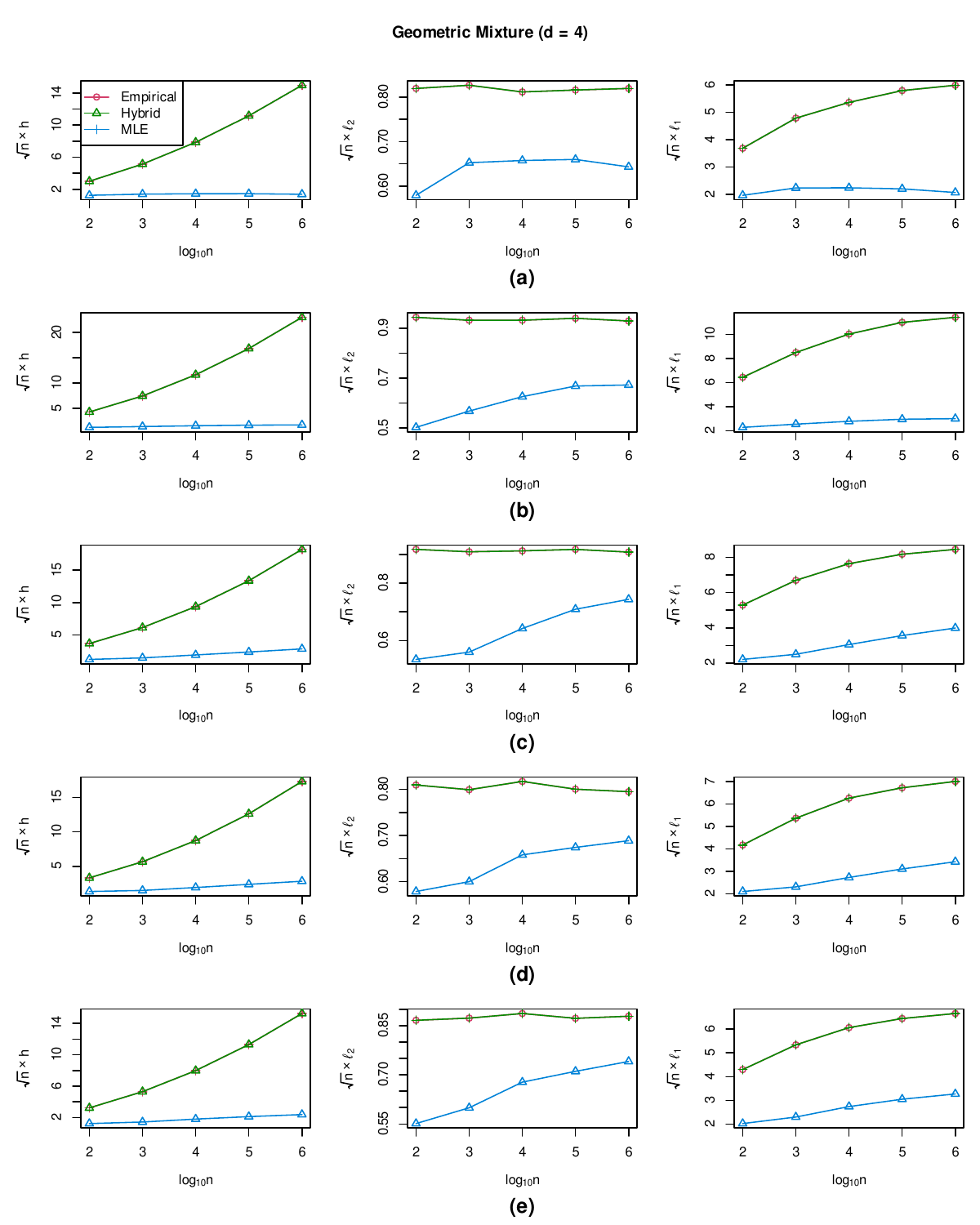}
  \caption{Four-dimensional mixtures of Geometric, for the same mixing
    distributions.} 
  \label{simgeo4d}
\end{figure*}

\begin{figure*}[!tbh]
  \centering
  \includegraphics[width=1\textwidth]{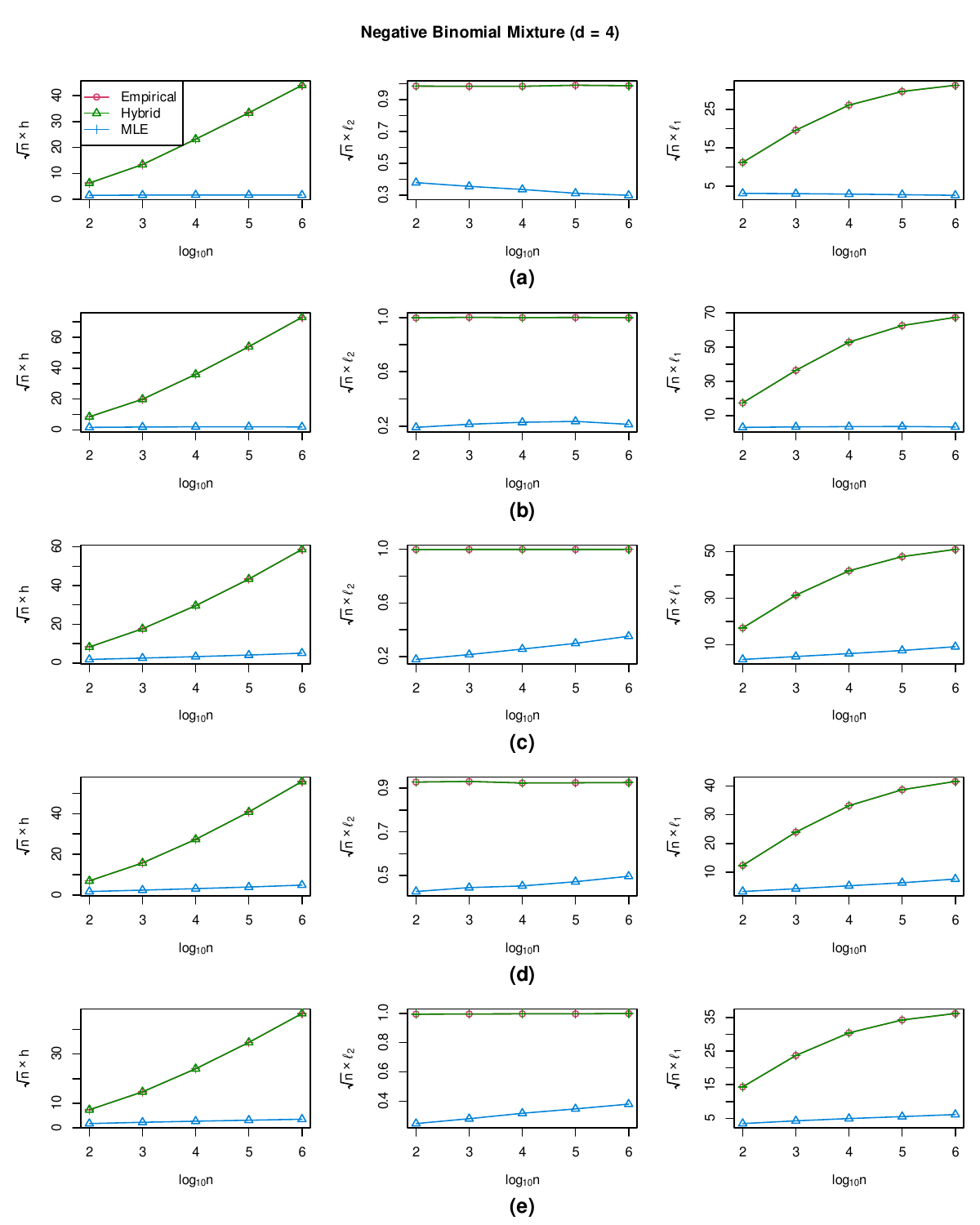}
  \caption{Four-dimensional mixtures of Negative Binomial, for the
    same mixing distributions.} 
  \label{simnb4d}
\end{figure*}

The simulation results are summarized and presented in
Figures~\ref{simpoi2d}--\ref{simgeo4d}. For both $d=2, 4$ we apply the
same mixture configurations for Poisson, Geometric and Negative
Binomial mixtures. Thus, we describe the setting only for $d=2$.  In
configuration (a), the true mixing distribution $Q_0$ has two support
points, one at $(0.7, 0.7)$ and another one at $(0.9, 0.9)$, with
masses $1/3$ and $2/3$, respectively. In (b), it has four support
points: $(0.6, 0.6)$, $(0.7, 0.7)$, $(0.8, 0.8)$ and $(0.9, 0.9)$,
with masses $1/10$, $2/10$, $3/10$, $4/10$, respectively. In (c),
$Q_0$ is the uniform distribution on ${[0.6, 0.9]}^2$. For
computational reasons, the uniform distribution is discretized to have
$11 \times 11$ support points. In (d), $Q_0$ has $1/3$ mass at
$(1, 1)$ and $2/3$ mass for the uniform distribution on
${[0.6, 0.9]}^2$. Finally, in configuration (e), the mixing
distribution has $1/3$ mass for $0.7 \times \mathcal{U}[0.6, 0.9]$ and
$2/3$ mass for $0.9 \times \mathcal{U}[0.6, 0.9]$. Here, the uniform
distribution $\mathcal{U}[0.6, 0.9]$ is discretized to have $101$
support points.

In all the settings considered here, the results confirm our
theoretical findings presented above. In the $\ell_1$- and the
$\ell_2$-distance, the hybrid estimator shows more or less the same
behavior as the empirical estimator, and hence we can certainly
conclude that it is $n^{-1/2}$-consistent. The MLE shows even a better
asymptotic behavior in these distances, suggesting that it is also
$n^{-1/2}$-consistent. For the Hellinger distance, the hybrid
estimator performs a little better than the empirical estimator, at
least for $d=2$, but the estimation error seems to blow up for large
sample sizes. The MLE, in contrast, shows a $n^{-1/2}$-consistency
behavior in the Hellinger distance.  However, we believe that the
convergence rate of the MLE in the Hellinger must include a
logarithmic factor. This is strongly suggested by the minimax lower
bounds discussed in \cite{FadouaHaraldYong} in the uni-dimensional
case.

\subsection{Real data application}
\label{sec:real-data}

For a real-world application, we consider the V\'elib data set that is
available in the \textsf{R} package \texttt{MBCbook}
\cite{bouveyron-celeux-etal-2019b}. It contains the numbers of
available bikes at 1213 stations in the ``V\'elib'' bike sharing
system in Paris, at every hour from 11 a.m. Sunday 31 August to 11
p.m. Sunday 7 September 2014. The data have been studied previously by
other researchers, using Poisson mixtures, often under the assumption
of conditional independence. Here, we study the relative performance
of our three estimators: The empirical, the hybrid and the
non-parametric maximum likelihood estimators. Later in
Section~\ref{testsection}, we will use the V\'elib data set again to
test the hypothesis of conditional independence.

We would like to consider a case where the assumption of conditional
independence should hold. Hence, we use the V\'elib data recorded at
12 p.m. Saturday 6 September and 12 p.m. Sunday 7 September because
for the data between these two time points the temporal correlation
should likely be negligible, if any. To investigate the performance of
the estimators, a $2$-fold cross-validation is used, where the dataset
is randomly split into two (roughly) equal-sized subsets: One is used
to compute the estimators, and the other one to produce an independent
empirical distribution for evaluating the performance measures of the
estimators. Three performance measures (not scaled by $\sqrt n$) are
calculated: the Hellinger, the $\ell_1$- and the
$\ell_1$-distances. To increase accuracy, the $2$-fold
cross-validation is repeated $1000$ times, and the overall means of
the performance measures are given in Table~\ref{Table-cv-velib}.

\begin{table}
  \centering
  \begin{tabular}{|c|c|c|c|} \hline
     & Hellinger & $\ell_2$-dist. & $\ell_1$-dist.\\ \hline
    Empirical & 0.677 & 0.0571 & 1.084 \\
    Hybrid & 0.677 & 0.0571 & 1.084 \\
    MLE & 0.571 & 0.0432 & 1.007 \\
    \hline
  \end{tabular}
  \caption{Cross-validation results for fitting a Poisson mixture to a
    two-dimensional V\'elib data subset.}
  \label{Table-cv-velib}
\end{table}

From the results of Table \ref{Table-cv-velib}, we observe that the
empirical estimator and the hybrid estimator show a similar behavior,
while the MLE exhibits clearly a superior performance.

\section{Testing for conditional independence}
\label{testsection}

In this section, we introduce a testing procedure to determine if the
conditional independence assumption holds or not. This testing
procedure, which is based on the bootstrap, will be later applied for
multivariate mixtures of Poisson and Geometric, with varying levels of
dependence.  Finally, we use this testing procedure to investigate
whether conditional independence holds for the V\'{e}lib dataset,
introduced in the previous section.

\subsection{A test for conditional independence}

Let us now explain the testing procedure. Fix some level
$\alpha \in (0,1)$. Suppose we observe $d$-dimensional data
$\mathbf{X}_1,..., \mathbf{X}_n$. Based on these observations, we
compute the non-parametric MLE $\widehat{\pi}_n$ under conditional
independence. We also calculate the empirical estimator
$\overline{\pi}_n$. Denote by $D_n$ some distance between
$\widehat{\pi}_n$ and $\overline{\pi}_n$, which could be the
Hellinger, the $\ell_1$- or the $\ell_2$-distance. In the simulations
presented below, we will always take $\alpha= 0.05$ and consider all
these three distance measures.

Now, choose a (large) integer $ B > 0$, and repeat the following procedure for $b=1,\ldots,B$:
\begin{itemize}
\item Generate i.i.d. $d$-dimensional random vectors
  $\mathbf{X}^{(b)}_1, ..., \mathbf{X}^{(b)}_n$ from
  $\widehat{\pi}_n$.
\item Based on these new data
  $\mathbf{X}^{(b)}_1, ..., \mathbf{X}^{(b)}_n$, compute again the MLE
  under conditional independence and the empirical estimator. Denote
  them by $\widehat{\pi}^{(b)}_n$ and $\overline{\pi}^{(b)}_n$,
  respectively.
\item Compute the same distance measure as above, but now between
  $\widehat{\pi}^{(b)}_n$ and $\overline{\pi}^{(b)}_n$. Denote the
  result by $D^{(b)}_n$.
\end{itemize}
This now leads to the $B$-sample $D^{(1)}_n,\ldots, D^{(B)}_n$. If the
conditional independence assumption holds true, we would expect this
sample to behave similarly as $D_n$. Thus, we will reject the
assumption of conditional independence if $D_n$ is larger than the
$(1-\alpha)$-quantile of the empirical distribution of
$D^{(1)}_n,..., D^{(B)}_n$.

\subsection{Simulations}

We now apply this testing procedure to two-dimensional mixtures of the
Poisson and the Geometric distribution, with varying levels of
dependence.

In all the simulations, $n=1000$. For the Poisson case, we proceed as
follows.  Let $Z \sim \textrm{Poi}(\beta \lambda)$ for
$\beta \in [0,1], \lambda > 0$. Also, let $Z_1$ and $Z_2$ be
independent, with $Z_i \sim \textrm{Poi}((1 - \beta) \lambda)$, for
$i=1,2$. Define $Y_1 := Z_1 + Z$, and $Y_2 := Z_2 + Z$. Then, it is a
well-known fact that the two-dimensional vector
$\mathbf Y := (Y_1,Y_2)$ is a bivariate Poisson, and that marginally,
$Y_i, i =1, 2$ follows a $\textrm{Poi}(\lambda)$-distribution.  The
case $\beta =1$ is degenerate in the sense that $Y_1= Y_2 = Z$, and
hence it is not covered in our investigation.

Note that as $\beta$ gets larger, the model is increasingly less
conditionally independent. Also, it is exactly conditional independent
when $\beta = 0$.  We consider a mixture with two components, with
means $(2,2)$ and $(4,4)$ and proportions $2/3$ and $1/3$,
respectively. The power is calculated at the level of $\alpha = 0.05$
based on the results of 1000 repetitions of $B =1000$ bootstrap
replications of the test procedure described above, using
$\beta = 0, 0.2, 0.4, 0.6, 0.8$. Thus, for each $\beta$, the algorithm
runs $1000 \times 1000$ times. Table \ref{Table1} shows the estimates
of the power when $D_n$ is the Hellinger, the $\ell_1$- or the
$\ell_2$-distance.

To construct a dependent bivariate Geometric distribution, we apply the following procedure. The main idea here is that if $F$ is the cdf of the Geometric distribution and $C$ is a given bivariate copula function, then $C \circ F$ defines the cdf of a bivariate Geometric  distribution in the sense that its marginal distributions are univariate Geometric.\\
For any parameter vector $\pmb \theta = (\theta_1,\theta_2)$, the
approach goes as follows. First, fix a dependence parameter
$\lambda > 1$. Let
$$C(u_1,u_2) := \exp (- ( (\log u_1)^\lambda + (\log u_2)^\lambda
)^{1/\lambda} )$$ be the Gumbel copula function, from which we
generate a vector $(u_1,u_2)$. To do so, we use the following steps:
\begin{itemize}
\item Generate two independent uniform random variables $(v_1,v_2)$.
\item Set $w (1- \log(w)/\lambda) = v_2$, and solve numerically for
  $w \in (0,1)$.
\item Set $u_1 := \exp (v_1^{1/\lambda} \log(w))$ and
  $u_2 := \exp ((1-v_1)^{1/\lambda} \log(w))$.
\end{itemize}
Now, for $i =1, 2$, set $Y_i := F^{-1}_{\theta_i}(u_i)$, where
$F_{\theta_i}$ is the cdf of a Geometric random variable with
parameter $\theta_i$.  Now, we have generated a random vector
$\mathbf Y := (Y_1,Y_2)$ whose marginal distributions are univariate
Geometric.

The dependence parameter $\lambda$ is a straightforward way to model
dependence. If $\lambda = 1$, then the components of the bivariate
vector are independent. So if our test works well, it should more
likely reject the null hypothesis when $\lambda$ is chosen larger.  We
set the success probabilities of our two-dimensional mixture to
$(0.7, 0.7)$ and $(0.9, 0.9)$, with masses $1/3$ and $2/3$,
respectively. As for the dependence parameter, we choose
$\lambda = 1, 1.25, 1.5, 1.75, 2$. The simulation setup is similar to
the one for Poisson mixtures described above, i.e., $M=1000$
repetitions of the $B=1000$ bootstrap test, leading to
$1000 \times 1000$ runs in total. The results are shown in Table
\ref{Table2}, again with level $\alpha=0.05$ and the Hellinger,
$\ell_1$- and $\ell_2$-distances as test statistic.

We conclude that the testing procedure gives very satisfactory
results. When the dependence gets stronger, then the power of the test
increases, as it should. This holds as well for Poisson as for
Geometric mixtures, and it also holds for all three distances. For
highly dependent mixtures (i.e., $\beta \ge 0.6$ in the Poisson case
or $\lambda \ge 1.5$ in the case of Geometric mixtures), the bootstrap
test has rejection rates of around $90\%$ or more.

\begin{table}
  \centering
  \begin{tabular}{|c|c|c|c|ccccc} \hline
     $\beta$ & Hellinger & $\ell_2$-dist. & $\ell_1$-dist.\\ \hline
    $0.0$ & 0.031 & 0.057 & 0.053 \\
    $0.2$ & 0.473 & 0.451 & 0.452 \\
    $0.4$ & 0.817 & 0.763 & 0.785 \\
    $0.6$ & 0.965 & 0.948 & 0.954\\
    $0.8$ & 0.992 & 0.993 & 0.993\\
    \hline
  \end{tabular}
  \caption{Power results of the bootstrap test for two-dimensional
    Poisson mixtures.}
  \label{Table1}
\end{table}

\begin{table} 
  \centering
  \begin{tabular}{|c|c|c|c|ccccc} \hline
     $\lambda$ & Hellinger & $\ell_2$-dist. & $\ell_1$-dist.\\ \hline
    $1.00$ & 0.009 & 0.014 & 0.016 \\
    $1.25$ & 0.102 & 0.347 & 0.269 \\
    $1.50$ & 0.892 & 0.993 & 0.986 \\
    $1.75$ & 1.000 & 1.000 & 1.000 \\
    $2.00$ & 1.000 & 1.000 & 1.000 \\
    \hline
  \end{tabular}
  \caption{Power results of the bootstrap test for two-dimensional
    Geometric mixtures.}
  \label{Table2}
\end{table}

\subsection{Application to the V\'elib data}

We will use again the V\'elib dataset with the aim of illustrating the
bootstrap test described in the previous section. Since there should
likely be a temporal correlation among the number of available bikes,
we use the bootstrap test to investigate the conditional independence
condition. We study two scenarios: The first one is for comparing the
numbers of available bikes at 1 a.m. and 5 a.m. Monday (1 September),
while in the second one, we compare those at 1 p.m. and 5
p.m. Monday. The two scenarios are chosen because we believe that
there should be a very strong temporal correlation at night but not so
much during the day, as the biking activity level is low at night but
high during the day. The dataset in either scenario is therefore
two-dimensional, with 1213 observations.

In each scenario, a bootstrap-estimated distribution of the distance
measures related to the Hellinger, $\ell_1$ and $\ell_2$-sense is
obtained, and Table~\ref{Tab:boot-velib} gives the $5$-number summary
of each distribution. In the first scenario, the three statistic
values are computed from the data, being $0.485, 0.0402, 0.837$,
respectively, all of which correspond to a $p$-value of $0$. This
indicates a high-level temporal correlation. In the second scenario,
we obtain $0.439, 0.0251, 0.657$, with $p$-values equal to
$0.621, 0.349, 0.501$, respectively. This means that the conditional
independence assumption cannot be rejected. Note that in both cases,
the results match quite well our expectations.

\begin{table}
  \centering
  \begin{tabular}{|c||c|c|c||c|c|c|}
    \hline & Hellinger & $\ell_2$-dist. & $\ell_1$-dist. & Hellinger
    & $\ell_2$-dist. & $\ell_1$-dist.\\
    \hline & \multicolumn3{c||}{1 a.m. vs.\@ 5 a.m.} & \multicolumn3{c|}{1 p.m. vs.\@ 5 p.m.} \\ \hline
    Min.    & 0.419 & 0.0236 & 0.640 & 0.413 & 0.0213 & 0.593 \\
    1st Qu. & 0.441 & 0.0258 & 0.694 & 0.436 & 0.0237 & 0.644 \\
    Median  & 0.446 & 0.0264 & 0.707 & 0.442 & 0.0246 & 0.658 \\
    3rd Qu. & 0.451 & 0.0272 & 0.720 & 0.447 & 0.0255 & 0.672 \\
    Max.  & 0.473 & 0.0326 & 0.774 & 0.465 & 0.0325 & 0.719 \\ \hline
  \end{tabular}
  \caption{Summaries of the statistic distributions estimated by bootstrap.}
  \label{Tab:boot-velib}
\end{table}

For the subset data used earlier in Section~\ref{sec:real-data}, we
also applied the above bootstrap test. The obtained $p$-values are
$0.600, 0.772, 0.784$ for using the three distances
respectively. Clearly one cannot reject the null hypothesis of
conditional independence for the two variables used, that is, the two
time points of 12 p.m. Saturday and 12 p.m. Sunday.

\section{Conclusions}

In this paper, we showed that for a wide range of multivariate
mixtures of PSDs with the conditional independence structure, the
non-parametric MLE converges to the truth in the Hellinger distance at
a rate that is very close to parametric. Although we believe that the
logarithmic factor in the rate cannot be improved (see also the
minimax rates and discussion in \cite{FadouaHaraldYong}), our
simulation results strongly suggest that the MLE converges at the
$n^{-1/2}$-rate in the $\ell_p$-distances for all $p \in [1,\infty]$
as it performs much better than the empirical and hybrid estimators
(which are both $n^{-1/2}$-consistent).  We believe that our results
are novel as, to the best of knowledge, it is the first time that a
paper presents the convergence rate of the MLE in multivariate
discrete mixtures as a function of the sample size $n$ and dimension
$d$, where the latter is allowed to grow in $n$.

As stated in the introduction, the conditional independence is a
simple way of making a multivariate mixture model
parsimonious. However, it is clear that one should first investigate
the validity of this assumption for an accurate inference. For this
reason, we introduced a testing procedure based on a bootstrap
approach. Based on our simulation study, we find that the test has
very good properties, including a high power under fixed alternatives.

We believe that the road we have taken here in investigating the
convergence rate of the MLE as well as implementing of the bootstrap
test, under conditional independence, was relatively well paved thanks
to our previous work on the MLE of one-dimensional mixtures of
PSDs. Having said that, we also believe that it would be possible to
extend some of the techniques used in this work to other dependence
structures. This can be achieved using copulas as done in Section 5.2
for the bi-variate Geometric distribution constructed with the help of
the Gumbel copula. The main challenge is that one might need to work
with the CDFs of PSDs instead of their pmfs.

Proving that the MLE is $n^{-1/2}$-consistent in the $\ell_1$- or at
least $\ell_2$-distance is a very interesting and difficult research
problems. The authors have spent quite some time exploring different
ideas to construct a proof but still without success. The main issue
is that it is very difficult to relate the Hellinger distance to
$\ell_1$ or $\ell_2$ distances in a way that the logarithm factor
disappears. In this sense, it seems to us that the hybrid estimator,
which puts the MLE and empirical estimators back to back, has the
potential of opening new theoretical possibilities.

\section*{Acknowledgments}

This work was financially supported by the Swiss National Fund Grant
(200021191999).

\appendix

\section*{Appendix}

In the following, we present the proofs that were left out in the main
manuscript.

\begin{theorem}
  \textbf{Existence and uniqueness of the MLE.} Let the true mixture
  be defined as
\begin{eqnarray*}
  \pi_0(\mathbf k) = \int_\Theta \prod_{j=1}^d f_{\theta_j}(k_j) dQ_0(\theta_1,\ldots,\theta_d),
\end{eqnarray*}
with $\mathbf k=(k_1,\ldots,k_d)$ and $Q_0$ denoting the unknown true
mixing distribution. Then, the corresponding non-parametric maximum
likelihood estimator (MLE) $\widehat{\pi}_n$ exists and is unique.
\end{theorem}

\medskip

\begin{proof}
  Let $\mathcal{T} = [0,R]$ if $b(R) < \infty$ and
  $\mathcal{T} = [0,R)$ if $b(R) = \infty$, and set
  $\Theta = \mathcal{T}^d$. Denote by $\mathcal Q$ the set of all
  mixing distributions defined on $\Theta$. Set now
  $\pmb \theta := (\theta_1,\ldots,\theta_d) \in \Theta$ and
  $\mathfrak{f}_{\pmb \theta}(\mathbf k) := \prod_{j=1}^d
  f_{\theta_j}(k_j)$, so that
\begin{eqnarray*}
  \pi_0(\mathbf k) = \int_\Theta \prod_{j=1}^d f_{\theta_j}(k_j) dQ_0(\theta_1,\ldots,\theta_d) 
  =  \int_\Theta \mathfrak{f}_{\pmb \theta}(\mathbf k) dQ_0(\pmb \theta).
\end{eqnarray*}
Let $\mathbf{X}_1, \ldots, \mathbf{X}_n$ be
i.i.d. $\mathbb{R}^d$-valued random variables distributed according to
$\pi_0$. We denote by $\mathbf k^1, \ldots, \mathbf k^U$ the distinct
values in $\mathbb{R}^d$ taken by the observations and set
$n_u = \sum_{i=1}^n \mathbb{I}_{\{\mathbf{X}_i = \mathbf k^u
  \}}$. With $Q \in \mathcal{Q}$, the likelihood function is then
given by
\begin{eqnarray*}
  L(Q) = \prod_{i=1}^n \int_\Theta \mathfrak{f}_{\pmb \theta}(\mathbf{X}_i) dQ(\pmb \theta)
  =  \prod_{u=1}^U \left(\int_\Theta  \mathfrak{f}_{\pmb \theta}(\mathbf k^u) dQ(\pmb \theta)  \right)^{n_u}.
\end{eqnarray*}
For the true mixing distribution $Q_0$, the likelihood function
$L(Q_0)$ is surely strictly positive, implying that the set
\begin{eqnarray*}
  \mathcal{M}  = \left \{  \left(L^1(Q),  \ldots, L^U(Q)  \right):  \ \ Q \in \mathcal Q  \right \}
\end{eqnarray*}
contains at least one interior point with strictly positive
likelihood. Here,
\begin{eqnarray*}
  L^u(Q) =   \left(\int_\Theta  \mathfrak{f}_{\pmb \theta}(\mathbf k^u) dQ(\pmb \theta)  \right)^{n_u}, \  \ u \in \{1, \ldots, U \}.
\end{eqnarray*}
We define the likelihood curve (including the null vector in $\mathbb{R}^U$) by
\begin{eqnarray*}
  \Gamma :=  \left\{ \Big(\mathfrak{f}_{\pmb \theta}(\mathbf k^1),\ldots,\mathfrak{f}_{\pmb \theta}(\mathbf k^U)\Big) : \vartheta \in \Theta \right\} \cup \Big\{ \big(0,\ldots,0\big)  \Big \}.
\end{eqnarray*}
Now we show that $\Gamma$ is a compact subset of $\mathbb R^U$. It is
clearly bounded since for all
$\mathbf{v} = (v^1, \ldots, v^U) \in \Gamma$, we have that
$\max_{1 \le u \le U} \vert v^u \vert \le 1$. But it is also
closed. Consider a sequence
$\mathbf{v}^{(l)}: = (v^{(l),1}, \ldots, v^{(l),U}) \in \Gamma$ such
that
\begin{eqnarray*}
  \lim_{l \nearrow  \infty} \mathbf{v}^{(l)}  = \mathbf{\widetilde{v}} =  (\widetilde{v}^1, \ldots, \widetilde{v}^U).
\end{eqnarray*}
If $\widetilde{v}^u = 0$ for all $u \in \{1, \ldots, U\}$, then the limit $\mathbf{\widetilde{v}}$ is clearly in $\Gamma$. Suppose now that there exists at least one index $u_0 \in \{1, \ldots, U\}$  such that $\widetilde{v}^{u_0}  \ne 0$. By definition of $\Gamma$, we can find a sequence $\pmb{\theta}^{(l)}$ such that $v^{(l),u} =  \mathfrak{f}_{\pmb{\theta}^{(l)}}(\mathbf k^u)$ for all $u \in \{1, \ldots, U\}$.\\
Consider first the case $R = \infty$. By contradiction, suppose that
the sequence $\pmb{\theta}^{(l)}$ is unbounded. This implies that
there exists a subsequence $\pmb{\theta}^{(l')}$, together with a
coordinate $j \in \{1,\ldots,d\}$, such that
$\lim_{l' \nearrow \infty} \theta_j^{(l')} = \infty$. But for any
fixed $k_j \in \mathbb N$, we have that
\begin{eqnarray*}
  \lim_{l'  \nearrow \infty} f_{\theta_j^{(l')}}(k_j)  =  \lim_{l'  \nearrow \infty} \frac{b_{k_j} (\theta_j^{(l')})^{k_j}}{b(\theta_j^{(l')})} \le \lim_{l'  \nearrow \infty} \frac{b_{k_j}}{b_{k_j+1} \theta_j^{(l')}} = 0,
\end{eqnarray*}
using that
$b(\theta_j^{(l')}) \ge b_{k_j+1} (\theta_j^{(l')})^{k_j+1}$. This
implies that
$\lim_{l' \nearrow \infty} \mathfrak{f}_{\pmb{\theta}^{(l')}}(\mathbf
k^{u_0}) = 0$, which contradicts our assumption above.  Thus,
$\pmb{\theta}^{(l)}$ has to be bounded. This now means that there
exists a subsequence $\pmb{\theta}^{(l')}$ and a
$\widetilde{\pmb{\theta}} $ such that
\begin{eqnarray*}
\lim_{l' \nearrow \infty}  \pmb{\theta}^{(l')} = \widetilde{\pmb{\theta}}.
\end{eqnarray*}
The map $\vartheta \mapsto \mathfrak{f}_{\pmb{\theta}}(\mathbf k)$ is
continuous, for any fixed $\mathbf k \in \mathbb N^d$ (at
$\pmb{\theta}=(0,\ldots,0) \in \mathbb{R}^d$, it is at least
right-continuous). Hence,
\begin{eqnarray*}
  \left(\mathfrak{f}_{\pmb{\theta}^{(l')}}(\mathbf k^1), \ldots, \mathfrak{f}_{ \pmb{\theta}^{(l')}}(\mathbf k^U)  \right)  \to  \left( \mathfrak{f}_{\widetilde{\pmb{\theta}}}(\mathbf k^1), \ldots,  \mathfrak{f}_{\widetilde{\pmb{\theta}}}(\mathbf k^U) \right)
\end{eqnarray*}
as $l' \nearrow \infty$, which implies
\begin{eqnarray*}
  (\widetilde{v}^1, \ldots, \widetilde{v}^U) =  \left( \mathfrak{f}_{\widetilde{\pmb{\theta}}}(\mathbf k^1), \ldots,  \mathfrak{f}_{\widetilde{\pmb{\theta}}}(\mathbf k^U) \right)
\end{eqnarray*}
by uniqueness of the limit. Therefore, we have shown that $(\widetilde{v}^1, \ldots, \widetilde{v}^U) \in \Gamma$.\\
Now consider the case $R < \infty$. Suppose first that
$b(R) = \infty$. We use the same notation as above. Again, we only
have to look at the case where there exists $u_0 \in \{1, \ldots, U\}$
such that $\widetilde{v}^{u_0} \ne 0$. We have
$\Theta \subset \overline \Theta = [0,R]^d$, and $\overline \Theta$ is
compact. Hence, the sequence $\pmb{\theta}^{(l)}$ has a subsequence
$\pmb{\theta}^{(l')}$ which converges to some
$\widetilde{\pmb{\theta}} = (\widetilde \theta_1,\ldots,\widetilde
\theta_d) \in \overline \Theta$. Suppose by contradiction that there
exists a coordinate $j \in \{1,\ldots,d\}$ such that
$\widetilde \theta_j = R$. Since
$\lim_{ \theta \nearrow R} f_\theta(k_j) = 0$, for any fixed
$k_j \in \mathbb N$, we reach a contradiction to our assumption
above. Hence, $\widetilde \vartheta$ is strictly smaller than $R$ in
all coordinates, which means that $\widetilde \vartheta \in
\Theta$. As before, we conclude using continuity of the map
$\vartheta \mapsto \mathfrak{f}_{\pmb{\theta}}(\mathbf k)$ and
uniqueness of the limit. For the case that $b(R) < \infty$, the
argument is even simpler because then, $\overline \Theta = \Theta$.

Hence, we have shown that $\Gamma$ is compact. Hence, we are in position to apply Theorem 18 in Chapter 5 of \cite{lindsay1995} plus the subsequent remark that one may include the zero vector in the likelihood curve since it can never appear in the maximizer. This implies that the MLE $\widehat{Q}_n \in \mathcal Q$ exists. The existence of $\widehat{\pi}_n$ then follows just by definition, which concludes the proof.
\end{proof}

\medskip

\begin{proposition}\label{Identifiability}
  \textbf{Identifiability of $Q_0$.}  Under Assumption (A1), the
  mixing distribution $Q_0$ in (\ref{Model}) is identifiable.
\end{proposition}

\smallskip

\begin{proof}
  since $\mathbb K = \mathbb N$, the condition
  $\sum_{k=1}^\infty k^{-1} = \infty$. Thus, we will follow the same
  approach in \cite{Patilea2005} used in the proof of Proposition
  1. Let $Q_1$ be another mixing distribution such that
\begin{eqnarray*}
  \pi_0(\mathbf k) = \int_{\Theta} \prod_{j=1}^d f_{\theta_j}(k_j) dQ_0(\theta_1, \ldots, \theta_d)  =  \int_{\Theta} \prod_{j=1}^d f_{\theta_j}(k_j) dQ_1(\theta_1, \ldots, \theta_d)
\end{eqnarray*}
for all $\mathbf k = (k_1, \ldots, k_d) \in \mathbb N^d$. By
Assumption (A1), $Q_0$ is supported on $[0, \tilde \theta]^d$. Suppose
that there exist some $r \in \{1, \ldots, d\}$ and $a > 0$ such that
\begin{eqnarray}\label{AssumpQ1}
  \int_{[\tilde \theta + a, R)} \int_{\mathcal{T}^{d-1}} dQ_1(\theta_1, \ldots, \theta_r, \ldots, \theta_d) > 0. 
\end{eqnarray}
Then, for all $k_r \in \mathbb N$
\begin{eqnarray}\label{IneqBelow}
  \pi_0(\mathbf k) &  \ge   &   \int_{[\tilde \theta + a, R)} \int_{\mathcal{T}^{d-1}} f_{\theta_r}(k_r) \prod_{1 \le j \ne r \le d } f_{\theta_j}(k_j) dQ_1(\theta_1, \ldots, \theta_d) \notag \\
                   & = &   \int_{[\tilde \theta + a, R)} \int_{\mathcal{T}^{d-1}} \frac{b_{k_r} \theta_r^{k_r}}{b(\theta_r)}  \prod_{1 \le j \ne r \le d} f_{\theta_j}(k_j) dQ_1(\theta_1, \ldots, \theta_d) \notag  \\
                   & \ge &  D b_{k_r} (\tilde \theta + a)^{k_r}
\end{eqnarray}
where 
$$
0 < D = \int_{[\tilde \theta + a, R)} \int_{\mathcal{T}^{d-1}}
\frac{1}{b(\theta_r)} \prod_{1 \le j \ne r \le d} f_{\theta_j}(k_j)
dQ_1(\theta_1, \ldots, \theta_d)
$$
by assumption (\ref{AssumpQ1}). On the other hand, we have that 
\begin{eqnarray}\label{IneqAbove}
  \pi_0(\mathbf k) &  =  &   \int_{[0, \tilde \theta]^d} \frac{b_{k_r} \theta_r^{k_r}}{b(\theta_r)} \prod_{1 \le j \ne r \le d } f_{\theta_j}(k_j) dQ_0(\theta_1, \ldots, \theta_d) \notag \\
                   & \le &  b_{k_r} \tilde{\theta}^{k_r} \frac{1}{b(0)} \int_{[0, \tilde \theta]^{d} }\prod_{1 \le j \ne r \le d } f_{\theta_j}(k_j) dQ_0(\theta_1, \ldots, \theta_d) \notag  \\
                   & \le & b_0^{-1}  b_{k_r} \tilde{\theta}^{k_r}
\end{eqnarray}
for all $k \in \mathbb N$, using the fact that $b(0) = b_0$,
$f_{\theta_j}(k_j) \le 1$ and that $Q_0$ is a probability
distribution.  Since the inequalities in (\ref{IneqBelow}) and
(\ref{IneqAbove}) are in contradiction, we conclude that $Q_1$ must be
also supported on $[0, \tilde \theta]^d$.  Thus, for all
$k_1, \ldots, k_d \in \mathbb N$
\begin{eqnarray}\label{MomEq}
  \int_{[0, \tilde \theta]^d}  \theta^{k_1}_1 \ldots  \theta^{k_d}_d d\widetilde{Q}_0(\theta_1,\ldots, \theta_d)  =  \int_{[0, \tilde \theta]^d}  \theta^{k_1}_1 \ldots  \theta^{k_d}_d d\widetilde{Q}_1(\theta_1,\ldots, \theta_d)  \end{eqnarray}
where for $i = 0, 1$
\begin{eqnarray*}
  d\widetilde{Q}_i(\theta_1,\ldots, \theta_d)   = c_0^{-1} \prod_{j=1}^d b(\theta_j)^{-1} dQ_i(\theta_1,\ldots, \theta_d)  
\end{eqnarray*}
where
$$
c_0 = \int_{[0, \tilde \theta]^d} \prod_{j=1}^d b(\theta_j)^{-1}
dQ_0(\theta_1,\ldots, \theta_d) = \int_{[0, \tilde \theta]^d}
\prod_{j=1}^d b(\theta_j)^{-1} dQ_1(\theta_1,\ldots, \theta_d) =
\frac{\pi_0(0, \ldots, 0)}{b_0^d}.
$$
The equalities in (\ref{MomEq}) are equivalent to saying that if
$\mathbf T = (T_1, \ldots, T_d) \sim \widetilde Q_0$ and
$ \mathbf{R} = (R_1, \ldots, R_d) \sim \widetilde Q_1$, then
$\mathbf T$ and $\mathbf{R}$ have the same moments of any order; i.e.,
\begin{eqnarray*}
  \mathbb E_{\widetilde Q_0}\left[T_1^{k_1} \times \ldots \times T_d^{k_d}\right]  = \mathbb E_{\widetilde Q_1}\left[R_1^{k_1} \times \ldots \times \tilde R_d^{k_d}\right].
\end{eqnarray*}
This in turn implies that 
\begin{eqnarray*}
  \mathbb E_{\widetilde Q_0}[e^{t_1 T_1 + \ldots + t_d T_d}]  = \mathbb E_{\widetilde Q_1}[e^{t_1 R_1 + \ldots + t_d R_d}],
\end{eqnarray*}
for all $t_1, \ldots, t_d \in \mathbb R$, that is that the moment
generating functions of $\mathbf T$ and $\mathbf{R}$ are equal. Hence,
$\widetilde Q_0 = \widetilde Q_1$ and $Q_0 = Q_1$.

\end{proof}

\medskip

\begin{theorem}
  \textbf{The case of finite support.} Assume that the support set
  $\mathbb K$ of the underlying PSD family is finite, and denote
  $K:=\card(\mathbb K)$. Then, we have for any $L > 0$ that
  \begin{eqnarray*}
    P\Big(h(\widehat \pi_n, \pi_0\Big) > \frac{L}{\sqrt n})  \le \frac{C K^d}{L},
  \end{eqnarray*}
  for some universal constant $C > 0$. In particular, we have that
  \begin{eqnarray*}
  h(\widehat \pi_n, \pi_0)   = O_{\mathbb P} \left( \frac{1}{\sqrt n}\right).
  \end{eqnarray*}
\end{theorem}

\medskip

\begin{proof}
  We are interested in the class of functions
\begin{eqnarray*}
  \mathcal{G}(\delta)  :=  \left \{ \mathbf k \mapsto g(\mathbf k)  =  \frac{\pi(\mathbf k)  - \pi_0(\mathbf k)}{\pi(\mathbf k)  +  \pi_0(\mathbf k)}, \mathbf k \in \mathbb K^d:  h(\pi, \pi_0) \le \delta\right\}.
\end{eqnarray*}
It is easy to see that the $\nu$-bracketing entropy is bounded from
above by $K^d \log \left(\frac{c \delta}{\nu} \right)$, for some
constant $c > 0$ which depends only on the
$\inf_{\mathbf k \in \mathbb K^d} \pi_0(k) > 0$.  Thus,
\begin{eqnarray*}
  \widetilde{J}_B(\delta, \mathcal G, \mathbb P) \le  \int_0^\delta  \sqrt{1 +  K^d \log \left(\frac{c \delta}{u} \right)}  du  \le \delta  + K^{d/2}  \int_0^\delta \sqrt{\log \left(\frac{c \delta}{u} \right)}  du  \le C K^{d} \delta,
\end{eqnarray*}
for some constant $C > 0$ which depends only on $K$ and the
$\inf_{\mathbf k \in \mathbb K} \pi_0(k)$. Following the same lines as
for bounding the probability $P_2$ in the proof of Theorem
\ref{RateOfConvergence}, the result then follows.
\end{proof}

\medskip

\noindent \textbf{Proof of Lemma \ref{BracketingIntegral}.}\\
Inequality (4.4) in \cite{Patilea} implies that if
$\pi_0(\mathbf k) \ge \kappa_n$, for some threshold $\kappa_n > 0$, we
have for all $\mathbf \in \mathbb N^d$ that
\begin{eqnarray*}
 \frac{\vert \pi(\mathbf k)  - \pi_0(\mathbf k)  \vert}{\pi(\mathbf k) + \pi_0(\mathbf k)} \mathbb{I}_{\{\pi_0(\mathbf k) \ge \kappa_n \}}  \le  \frac{2 h(\pi, \pi_0)}{\sqrt{\kappa_n}}.
\end{eqnarray*}
Thus, for any element $g \in \mathcal{G}_n(\delta)$ and for all $\mathbf k  \in \{0, \ldots, K_n \}^d$, we have that
\begin{eqnarray*}
  g(\mathbf k) = \frac{\vert \pi(\mathbf k)  - \pi_0(\mathbf k)  \vert}{\pi(\mathbf k) + \pi_0(\mathbf k)} \mathbb{I}_{\{\pi_0(\mathbf k) \ge \tau_n \}}  \in \left[-\frac{2 \delta}{\sqrt{\tau_n}}, \frac{2 \delta}{\sqrt{\tau_n}}\right],
\end{eqnarray*}
with $\tau_n$ is the same quantity defined in (\ref{taun}). We now
partition this interval into $N$ equal sub-intervals of size $s$
(depending on $\delta$), which must satisfy
$s N = 4 \delta/\sqrt{\tau_n}$. For any
$\mathbf k \in \{0, \ldots, K_n\}^d$, there exists
$i_{\mathbf k} \in \{0, \ldots, N-1\}$ such that
\begin{eqnarray*}
  L_i(\mathbf k)  := -\frac{2 \delta}{\sqrt{\tau_n} } +  i_{\mathbf k} s   \le  g(\mathbf k)   \le    U_i(\mathbf k) :=  -\frac{2 \delta}{\sqrt{\tau_n}}  +  (i_{\mathbf k}+1) s.
\end{eqnarray*}
Note that
\begin{eqnarray*}
  \sum_{\mathbf k: \max_{1 \le j \ne d} k_j \le K_n} (U_i(\mathbf k)  -  L_i(\mathbf k))^2  \pi_0(\mathbf k)    =  s^2 \sum_{\mathbf k: \max_{1 \le j \ne d} k_j \le K_n}  \pi_0(\mathbf k)  \le s^2.
\end{eqnarray*}
Thus, we can take $\nu = s$ so that $[L_i(\mathbf k), U_i(\mathbf k)]$
is a $\nu$-bracket, implying that
\begin{eqnarray*}
  N =  \frac{4\delta}{\sqrt{\tau_n} \nu}.
\end{eqnarray*}
The number of brackets needed to cover $\mathcal{G}_n(\delta)$ is at
most $N^{(K_n+1)^d}$. Hence, an upper bound on the $\nu$-bracketing
entropy is given in the following inequality
\begin{eqnarray*}
  H_B(\nu,\mathcal{G}_n(\delta),\mathbb{P}) & \le & (K_n +1)^d  \log N   =    (K_n+1)^d  \log \left(\frac{4 \delta}{\sqrt{\tau_n} \nu}  \right)  \\
                                            &  \le  &  (K_n+1)^d \log 4  + \frac{1}{2} (K_n+1)^d  \log\left(\frac{1}{\tau_n}\right)   +  (K_n+1)^d  \log \left(\frac{\delta}{\nu}\right)  \\
                                            &   \le  &    (K_n+1)^d  \log\left(\frac{1}{\tau_n}\right)   +  (K_n+1)^d  \log \left(\frac{\delta}{\nu}\right)
\end{eqnarray*}
for $n$ large enough such that $\log 4 \le \log(1/\tau_n)/2$ or
equivalently $\tau_n \le 1/16$. Using
$\sqrt{x+y} \le \sqrt{x} + \sqrt y$ for all $x, y \in [0,\infty)$, we
get
\begin{eqnarray*}
  \int_{0}^\delta  H^{1/2}_B(u, \mathcal{G}_n(\delta), \mathbb P) du 
  & \le &   (K_n+1)^{d/2} \sqrt{\log\left(\frac{1}{\tau_n}\right) }  \delta   +  (K_n+1)^{d/2} \int_{0}^\delta   \sqrt{\log \left(\frac{\delta}{u}\right)} du.
\end{eqnarray*}
By elementary calculus, we can bound the second integral by
$\delta$. Hence, we obtain for $n$ large enough that
\begin{eqnarray*}
  \int_{0}^\delta  H^{1/2}_B(u, \mathcal{G}_n(\delta), \mathbb P) du & \le &  (K_n+1)^{d/2} \left( \sqrt{\log\left(\frac{1}{\tau_n}\right) }  \delta   +  \delta \right)\\
                                                                     & \le & 2\delta (K_n+1)^{d/2}  \sqrt{\log\left(\frac{1}{\tau_n}\right)}.
\end{eqnarray*}
Thus, for $n$ large enough, we obtain by definition of the bracketing
integral and the inequality $\sqrt{x+y} \le \sqrt{x} + \sqrt y$ that
\begin{eqnarray*}
  \widetilde{J}_B(\delta, \mathcal{G}_n(\delta), \mathbb P)
  & \le & \delta + \int_{0}^\delta  H^{1/2}_B(u, \mathcal{G}_n(\delta), \mathbb P) du 
          \le   3 \delta (K_n+1)^{d/2}  \sqrt{\log\left(\frac{1}{\tau_n}\right)} \\
  & \le & \frac{\sqrt d \cdot 3^{(5+d)/2}}{\log (1/t_0)^{1+d/2}} \log (nd)^{1+d/2} \delta,
\end{eqnarray*}
where in the last step Lemma \ref{Kntaun} was applied. \hfill $\Box$

\medskip

\noindent \textbf{Proof of Lemma \ref{Kntaun}.}\\
Let $U$ and $W$ be the same constants in (\ref{UW}). It follows from
property 3 of Lemma \ref{Lemma1} that for all $K \ge \max(U, W)$, we
have that
\begin{eqnarray}\label{Bound}
 \sum_{\mathbf k : \max_{1 \le j \le d} k_j \ge K+1}  \pi_0(k)  \le  A d t_0^K.
\end{eqnarray} 
Hence,
\begin{eqnarray*}
  \sum_{\mathbf k: \max_{1 \le j \le d} k_j \ge K+1}  \pi_0(\mathbf k)  \le \frac{\log (nd)^{2+d}}{n}
\end{eqnarray*}
provided that
\begin{eqnarray*}
  K    \ge     \frac{1}{\log(1/t_0)}  \log \left(\frac{A n d}{(\log (nd))^{2+d}} \right)  =   \frac{1}{\log(1/t_0)}  \Big( \log(A)  + \log(nd)  -  (2+d) \log(\log (nd))  \Big). 
\end{eqnarray*}
Let $n \ge A$. Then, 
\begin{eqnarray*}
  \frac{1}{\log(1/t_0)}  \Big( \log(A)  + \log(nd)  -  (2+d) \log(\log (nd))  \Big)  \le  \frac{ \log(n) + \log (nd)}{\log(1/t_0)}  \le   \frac{2\log (nd)}{\log(1/t_0)}.
\end{eqnarray*}
Thus,the tail bound in (\ref{Bound}) is satisfied if
$$
K \ge \frac{2\log (nd)}{\log(1/t_0)}.
$$
By definition of $K_n$ as the smallest integer $K$ satisfying (\ref{Bound}), we thus have
\begin{eqnarray*}
  K_n   \le \Big \lfloor  \frac{2 \log (nd)}{\log(1/t_0)} \Big \rfloor  +  1 =:  \widetilde{K}_n,
\end{eqnarray*}
which implies that for $n$ large enough
\begin{eqnarray}
  \label{Range1}
  \widetilde K_n \le \frac{3 \log (nd)}{\log (1/t_0)}.
\end{eqnarray}
We now move onto bounding the quantity $\log(1/\tau_n)$.  For $n$
large enough so that $\widetilde{K}_n \ge \max(U,V,W)$, where $V$ is
from Assumption (A3) we have by property 4 of Lemma \ref{Lemma1} that
\begin{eqnarray*}
  \tau_n  = \inf_{0 \le k_j \le K_n, \forall j=1,\ldots,d} \pi_0(\mathbf k)  \ge \pi_0(\widetilde K_n, \ldots, \widetilde K_n) & = &  \int_\Theta \prod_{j=1}^d f_{\theta_j}(\widetilde K_n) dQ_0(\theta_1,\ldots,\theta_d).
\end{eqnarray*}
Note that $\widetilde K_n \ge \max(U,V,W)$ if and only if
\begin{eqnarray}\label{Range2}
\Big \lfloor  \frac{2 \log (nd)}{\log(1/t_0)} \Big \rfloor \ge \max(U,V,W) -1.
\end{eqnarray}
Now, if $Q_0(\{0,\ldots,0\}) > 0$, it follows from Assumption (A2)
that $Q_0([\delta_0, R)^d) \ge \eta_0$.  Hence, using property $1$ of
Lemma \ref{Lemma1}, it follows that
\begin{eqnarray*}
\tau_n \ge \eta_0 f_{\delta_0}(\widetilde K_n)^d.
\end{eqnarray*}
In the case that $Q_0(\{0,\ldots,0\}) = 0$, we know the same
assumption that $Q_0((\delta_0, R)^d) = 1$. Invoking again property
$1$ of Lemma \ref{Lemma1}, we see that in any case
\begin{eqnarray*}
  \tau_n \ge \eta_0  f_{\delta_0}(\widetilde K_n)^d
  = \eta_0\left( \frac{b_{\widetilde{K}_n} \delta^{\widetilde K_n}_0}{b(\delta_0)} \right)^d
  \ge \eta_0  \left(\frac{b_0 \widetilde{K}_n^{-\widetilde{K}_n} \delta^{\widetilde K_n}_0}{ b(\delta_0)}\right)^d,
\end{eqnarray*}
where the last step applied Assumption (A3) (recall that we assume
that $\widetilde K_n \ge V$). Thus, we obtain for $n$ large enough
\begin{eqnarray}\label{Range3}
  \log(1/\tau_n) & \le & \log\left(\frac{b(\delta_0)^d}{b_0^d \eta_0} (\widetilde{K}_n^{\widetilde{K}_n} \delta^{-\widetilde K_n}_0)^d\right) \notag \\
                 & \le   &   \log \left(\frac{b(\delta_0)^d}{b_0^d\eta_0}\right) + d \widetilde{K}_n \log(\widetilde K_n) + d\widetilde{K}_n \log\left( \frac{1}{\delta_0}\right)  \notag \\
                 & \le   &  3 d \widetilde{K}_n \log(\widetilde K_n)
                           \le  3 d \widetilde{K}^{2}_n
                           \le   \frac{3^3 d  \log (nd)^{2}}{\log (1/t_0)^2},
\end{eqnarray}
implying that
\begin{eqnarray}\label{Range4}
  (K_n +1)^d \log(1/\tau_n)
  &  \le & \frac{3^3 d \log (nd)^{2}}{\log (1/t_0)^2} \left( \frac{2 \log(nd)}{\log(1/t_0)}  + 2\right)^d \notag \\
  & \le & \frac{3^3 d \log (nd)^{2}}{\log (1/t_0)^2} \left( \frac{3 \log(nd)}{\log(1/t_0)}\right)^d 
          \le   \frac{3^{3+d} d \log (nd)^{2 +d}}{\log (1/t_0)^{2+d} }.
\end{eqnarray}
Now, we will derive a lower bound for $n$ in order for the
inequalities (\ref{Range1}), (\ref{Range2}), (\ref{Range3}) and
((\ref{Range4}) to be fulfilled. It is easy to see that it is enough
that $n$ satisfies
\begin{eqnarray*}
  \frac{2 \log(nd)}{\log(1/t_0)}  + 1 \le \frac{3 \log(nd)}{\log(1/t_0)},
\end{eqnarray*}
\begin{eqnarray*}
  \frac{2 \log(nd)}{\log(1/t_0)}  \ge  \max(U, V, W),
\end{eqnarray*}
\begin{eqnarray*}
  \log\left(\frac{b(\delta_0)^d}{b^d_0 \eta_0}\right) \le d \log \left(\frac{2 \log(nd)}{\log(1/t_0)}\right), \ \ \textrm{and} \ \ \log\left(\frac{1}{\delta_0}\right)  \le d \log\left(\frac{2 \log(nd)}{\log(1/t_0)}\right),
\end{eqnarray*}
and 
\begin{eqnarray*}
  \frac{2 \log(nd)}{\log(1/t_0)}  + 2 \le \frac{3 \log(nd)}{\log(1/t_0)}.
\end{eqnarray*}
Solving for $n$ yields
\begin{eqnarray*}
  n  \ge \frac{1}{d} \cdot \frac{1}{t^2_0} \vee \exp\left \{\log\left(\frac{1}{\sqrt t_0}\right) \cdot \left(U \vee V \vee W \vee \frac{b(\delta_0)}{b_0 \eta_0^{1/d}} \vee \frac{1}{\delta_0^{1/d}}   \right)    \right \}.
\end{eqnarray*}
On the other hand, we know that we need $n \ge A \vee 3$. Using the
expression of $A$, and using the fact that $f_\theta \in [0,1]$, we
see that this inequality is satisfied if
\begin{eqnarray*}
  n \ge \frac{1}{t^{W-1}_0(1-t_0)}.
\end{eqnarray*}
Since $W \ge 3$, $W-1 \ge 2$ and hence
$1/t_0^{W-1} \ge 1/t_0^2 \ge 1/(d t_0^2)$.  It follows that we can
take
\begin{eqnarray*}
  n  \ge  \bigg \lfloor \frac{1}{d} \cdot  \exp\left \{\log\left(\frac{1}{\sqrt t_0}\right) \cdot \left(U \vee V \vee W \vee \frac{b(\delta_0)}{b_0 \eta_0^{1/d}} \vee \frac{1}{\delta_0^{1/d}}   \right)    \right \} \vee \frac{1}{t_0^{W-1}(1-t_0)} \bigg \rfloor +1:= N(d, t_0, \tilde \theta, \delta, \eta_0).
\end{eqnarray*}

\hfill $\Box$

\medskip

\medskip

\noindent \textbf{Proof of Proposition \ref{Boundfortildek}.}\\
Our convergence result for the MLE (Theorem \ref{RateOfConvergence})
tells us that
\begin{eqnarray*}
  \sum_{\mathbf k \in \mathbb{N}^d}  \vert \widehat \pi_n(\mathbf k) - \pi_0(\mathbf k) \vert = O_{\mathbb P}\Big(\frac{\log (nd)^{1+d/2}}{\sqrt{n}}\Big) = o_{\mathbb P}\Big(\frac{1}{ \log (nd)^{2+d}}\Big).
\end{eqnarray*}
This implies
\begin{eqnarray*}
  \sum_{\mathbf k: \max_{1 \le j \le d} k_j > \widetilde{K}_n}  \pi_0(\mathbf k)
  \le \sum_{\mathbf k \in \mathbb{N}^d}  \vert \widehat \pi_n(k) - \pi_0(k) \vert + \sum_{\mathbf k: \max_{1 \le j \le d} k_j > \widetilde{K}_n}  \widehat \pi_n(k)   \le \frac{2}{\log (nd)^{2+d}}.
\end{eqnarray*}
By property 3 of Lemma \ref{Lemma1} we know that for for
$K \in \mathbb N$ large enough
\begin{eqnarray}\label{IneqAt0}
  \sum_{\mathbf k: \max_{1 \le j \le d} k_j \ge K+1}  \pi_0(k)  \le A d t_0^K.
\end{eqnarray}
Let $K  > 0$ be such that $A d t_0^{K} \le \frac{2}{\log (nd)^{2+d}}$. Then,
\begin{eqnarray*}
  K \ge \frac{1}{\log (1/t_0)} \log\Big(\frac{Ad}{2} \log (nd)^{2+d}\Big) &=& \frac{1}{\log(1/t_0)} \log\Big(\frac{Ad}{2}\Big)+ \frac{2+d}{\log(1/t_0)}  \log(\log (nd)).
\end{eqnarray*}
Note that the term on the right of the latter display is
$\le \frac{3+d}{\log(1/t_0)} \log(\log (nd)) $ for $n$ large enough
and hence the inequality in (\ref{IneqAt0}) is satisfied for
$K > \frac{3+d}{\log(1/t_0)} \log(\log (nd))$. Thus, by definition of
$\widetilde{K}_n$, we have for large enough $n$
\begin{eqnarray*}
  \widetilde{K}_n  + 1  \le \frac{3+d}{\log(1/t_0)} \log(\log (nd))  +1  \le \frac{4+d}{\log(1/t_0)} \log(\log (nd))  =: N_d.
\end{eqnarray*}
Without loss of generality, we may assume that $N_d$ is an integer. In
addition, we assume in the sequel that $Q_0(\{0,\ldots,0\}) = 0$,
which means by Assumption (A2) that
$\supp{Q_0} \subset [\delta_0, R)^d$, for some $\delta_0 \in (0, R)$
(if $Q_0(\{0,\ldots,0\}) > 0$, a similar reasoning yields the same
conclusions).  Using property 1 and 4 of Lemma \ref{Lemma1} and
Assumption (A3), it follows that
\begin{eqnarray*}
  \Big(1-\pi_0(\widetilde{K}_n,\ldots,\widetilde{K}_n)\Big)^n  &\le & \Big(1-\pi_0(N_d,\ldots,N_d)\Big)^n   \\ 
                                                               &=& \left( 1- \int_\Theta \prod_{j=1}^d f_{\theta_j}(N_d) dQ_0(\theta) \right)^n\\
                                                               & \le &  \left(1-  f_{\delta_0}(N_d)^d \right)^ n  = \left( 1-  \left(\frac{b_{N_d} \delta_0^{N_d}}{b(\delta_0)}\right)^d  \right)^n = \left( 1-  \left(\frac{b_0}{b(\delta_0)} N_d^{-N_d} \delta_0^{N_d}\right)^d  \right)^n.
\end{eqnarray*}
Using the fact that $\log(1-x) \le -x$, for $x > 0$ it follows that 
\begin{eqnarray*}
  \Big(\widetilde{K}_n +1\Big)^d \Big(1-\pi_0(\widetilde{K}_n,\ldots,\widetilde{K}_n))\Big)^n  \le \exp(\psi_{n,d})  
\end{eqnarray*}
where
\begin{eqnarray*}
  \psi_{n,d}  & = &  d \log\left( \frac{4+d}{\log(1/t_0)} \right)   +  d \log\big(\log(\log(nd))\big)  - \ n \frac{b^d_0}{b(\delta_0)^d} \left( \frac{(4+d)\delta_0 \log(\log(nd))}{\log(1/t_0)} \right)^{-\frac{d (4+d) \log(\log(nd))}{log(1/t_0)}}.
\end{eqnarray*}
Now, note that 
\begin{eqnarray*}
  \lim_{n \to \infty } \sqrt n \left( \frac{(4+d)\delta_0 \log(\log(nd))}{\log(1/t_0)} \right)^{-\frac{d (4+d) \log(\log(nd))}{log(1/t_0)}} = \infty
\end{eqnarray*}
since 
\begin{eqnarray*}
  && \hspace{-2cm} \log\left(\sqrt n \left( \frac{(4+d)\delta_0
     \log(\log(nd))}{\log(1/t_0)} \right)^{-\frac{d (4+d)
     \log(\log(nd))}{log(1/t_0)}} \right) \\
  & = &  \frac{n}{2}  -    \frac{d (4+d) \log(\log(nd))}{log(1/t_0)}  
        \log \left( \frac{(4+d)\delta_0 \log(\log(nd))}{\log(1/t_0)}
        \right) \to \infty. 
\end{eqnarray*}
Hence, for $n$ large enough
$$
\psi_{n, d} \le d \log\left( \frac{4+d}{\log(1/t_0)} \right) + d
\log\big(\log(\log(nd))\big) - \sqrt n/2 \to - \infty
$$
implying that $\exp(\psi_{n, d} \to 0$. This concludes the proof.
\hfill $\Box$

\bibliographystyle{Chicago}

\end{document}